\documentclass[12pt,reqno]{amsart}
\usepackage{amsmath,amssymb,amsfonts,amsthm}
\usepackage[left=3cm, right=3cm, top=2.8cm, bottom=2.8cm]{geometry}
\usepackage[utf8]{inputenc}
\usepackage[T1]{fontenc}
\usepackage{color}
\usepackage{multirow}
\usepackage{graphicx}
\usepackage{hyperref}
\usepackage{refcheck}
\usepackage{graphicx}
\usepackage{calc}%
\usepackage{hyperref}
\hypersetup{
    colorlinks=true,
    linkcolor=blue,
    filecolor=magenta,
    urlcolor=cyan,
}

\usepackage{stackrel}

\DeclareMathOperator*{\esssup}{ess\,sup}
\newtheorem{theorem}{Theorem}
\newtheorem{lemma}[theorem]{Lemma}

\newtheorem{corollary}[theorem]{Corollary}
\newtheorem{definition}[theorem]{Definition}
\newtheorem{remark}[theorem]{Remark}

\theoremstyle{plain}

\newenvironment{proof1}{\vspace*{0.5cm}\paragraph{{\it Proof of Lemma 33.}}}{\hfill$\square$\vspace*{0.5cm}}

\begin{document}

\title[The Riemann-Liouville fractional integral]{The Riemann-Liouville fractional integral in Bochner-Lebesgue spaces II}


\author[P. M. Carvalho-Neto]{Paulo M. de Carvalho-Neto}
\address[Paulo M. de Carvalho Neto]{Department of Mathematics, Federal University of Santa Catarina, Florian\'{o}polis - SC, Brazil}
\email[]{paulo.carvalho@ufsc.br}
\author[R. Fehlberg J\'{u}nior]{Renato Fehlberg J\'{u}nior}
\address[Renato Fehlberg J\'{u}nior]{Department of Mathematics, Federal University of Esp\'{i}rito Santo, Vit\'{o}ria - ES, Brazil}
\email[Corresponding Author]{renato.fehlberg@ufes.br}


\subjclass[2010]{26A33, 47G10, 46B50}


\keywords{Riemann-Liouville fractional integral, Hardy-Littlewood Theorem, Bochner-Lebesgue spaces, compact operator}


\begin{abstract}
In this work we study the Riemann-Liouville fractional integral of order $\alpha\in(0,1/p)$ as an operator from $L^p(I;X)$ into $L^{q}(I;X)$, with $1\leq q\leq p/(1-p\alpha)$, whether $I=[t_0,t_1]$ or $I=[t_0,\infty)$ and $X$ is a Banach space. Our main result give necessary and sufficient conditions to ensure the compactness of the Riemann-Liouville fractional integral from $L^p(t_0,t_1;X)$ into $L^{q}(t_0,t_1;X)$, when $1\leq q< p/(1-p\alpha)$.
\end{abstract}

\maketitle

\section{Introduction}

In 1928, Hardy and Littlewood proved that Riemann-Liouville fractional integral of order $\alpha\in(0,1/p)$ is a bounded operator from $L^p(I;\mathbb{R})$ into $L^{p/(1-p\alpha)}(I;\mathbb{R})$, whether $I=[t_0,t_1]$ or $I=[t_0,\infty)$. To be more precise, let us enunciate their result bellow (see \cite[Theorem 4]{HaLi1} for details on the proof).
\begin{theorem}\label{theoHL}
Consider $p\in(1,\infty)$, $\alpha\in(0,1/p)$ and assume that $I$ denotes $[t_0,t_1]$ or $[t_0,\infty)$.  If $f\in L^p(I;\mathbb{R})$ we have that $J_{t_0,t}^\alpha f(t)$ belongs to $L^{p/(1-p\alpha)}(I;\mathbb{R})$ and that there exists $K>0$ such that
\begin{equation*}\left[\int_{I}{\left|J^\alpha_{t_0,s}f(s)\right|^{p/(1-p\alpha)}}\,ds\right]^{{(1-p\alpha)/p}}\leq K\left[\int_{I}{|f(s)|^p}\,ds\right]^{1/p}.\end{equation*}
The constant $K$ above only depends on $\alpha$ and $p$.
\end{theorem}

This theorem was one of the first results in the literature that helped us to better understand the behavior of the Riemann-Liouville fractional integral. Of course that this study is much more profound, since it gave birth to the classical Hardy-Littlewood-Sobolev inequality, which is broadly used to study Partial Differential Equations.

There were other researchers that gave some new proofs to Theorem \ref{theoHL}. We may bring light to O'Neil in \cite{ONe1} as a very interesting example. There he discuss some generalizations of H\"{o}lder's inequality and Young's Inequality for convolutions using Young's functions and the Orlicz spaces. The result from which we can derive Theorem \ref{theoHL} is \cite[Theorem 4.7]{ONe1}.

Nevertheless, in both papers (Hardy-Littlewood and O'Neil), the proof of Theorem \ref{theoHL} is complex. In fact, the absence of a more simple proof in the literature (as far as the authors are aware) was what motivated the first stage of the study we have made. Of course, it is worth to point that our main concern here is to deal with more general spaces (Bochner-Lebesgue spaces; see Definition \ref{definitial}), therefore, the study itself is essentially more general.

Our objective here is to give continuity to our later work \cite{CarFe0}. There, we have focused in the discussion of Riemann-Liouville fractional integral, when it is viewed as an operator from a Bochner-Lebesgue space into itself, and discussed some properties associated to it; like boundedness, compactness (the main topic of the work) and estimates of its norm. Let us recall a small improvement of Theorem \ref{theoHL} bellow.

\begin{theorem}[{\cite[Theorem 11]{CarFe0}}]\label{minkowskiseq} Let $\alpha>0$, $1\leq p \leq \infty$ and $f\in L^p(t_0,t_1;X)$. Then $J_{t_0,t}^\alpha f(t)$ is Bochner integrable and belongs to $L^p(t_0,t_1;X)$. Furthermore, it holds that
\begin{equation*}\left[\int_{t_0}^{t_1}{\left\|J_{t_0,t}^\alpha f(t)\right\|^p_X}\,dt\right]^{1/p}\leq \left[\dfrac{(t_1-t_0)^\alpha}{\Gamma(\alpha+1)}\right] \|f\|_{L^p(t_0,t_1;X)}.\end{equation*}
In other words, $J_{t_0,t}^\alpha$ is a bounded operator from $L^p(t_0,t_1;X)$ into itself.
\end{theorem}

It is important to note that Theorem \ref{minkowskiseq} cannot be proved in the context of unbounded intervals. In fact, the Riemann-Liouville fractional integral cannot even define a linear operator from $L^p(t_0,\infty;X)$ into itself (see \cite[Theorem 32]{CarFe0} for details).

In the same work, the authors also gave a necessary and sufficient condition for the compactness of $J_{t_0,t}^\alpha$ as an operator from $L^p(t_0,t_1;X)$ into itself, which was one of the main results proved there.

\begin{theorem}[{\cite[Theorems 15 and 19]{CarFe0}}]\label{compactrieman} Let $p\in[1,\infty]$ and $\alpha>0$. The bounded operator
$J_{t_0,t}^\alpha:L^p(t_0,t_1;X)\rightarrow L^p(t_0,t_1;X)$
is compact if, and only if, for any bounded set $F\subset L^p(t_0,t_1;X)$ it holds that
\begin{equation*}\left|\begin{array}{l}\left\{J_{t_0,t_1^*}^{1+\alpha} f(t_1^*)-J_{t_0,t_0^*}^{1+\alpha} f(t_0^*)\,:\,f\in F\right\}\textrm{ is relatively compact in } X,\vspace*{0.2cm}\\
\textrm{for every }t_0<t_0^*<t_1^*<t_1.\end{array}\right.\end{equation*}
\end{theorem}

From this results, a natural question that arises is about the existence of a greatest exponent $q\in[1,\infty]$ such that
$$J_{t_0,t}^\alpha:L^p(I;X)\rightarrow L^q(I;X),$$
is a well-defined linear operator. Moreover, we may argue if it is bounded and compact, regardless of whether $I=[t_0,t_1]$ or $I=[t_0,\infty)$, like it was done by Hardy and Littlewood when $X=\mathbb{R}$.

Even though variations of the operator induced by the Riemann-Liouville fractional integral have been extensively studied in the literature (when $X=\mathbb{R}$), there are no evidence of this results in the case of Bochner-Lebesgue spaces. In order to recall some of these classical works, we may cite Stepanov in \cite{Ste1} where he addressed the boundedness and compactness of operators $K:L^p(I;\mathbb{R})\rightarrow L^q(I;\mathbb{R})$, where
$$Kf(t)=v(t)\int_0^tK(t,s)u(s)f(s)ds,$$
with $v(t)$ and $u(t)$ locally integrable functions and $K(t,s)$ satisfying some particular conditions, or even Prokhorov in \cite{Pro1} where  he addressed the same question, however to the operators $T_\alpha:L^p((0,\infty);\mathbb{R})\rightarrow L^q((0,\infty);\mathbb{R})$, which were given by
$$T_\alpha f(t)=\dfrac{v(t)}{t^\alpha}\int_0^t(t-s)^{\alpha-1}f(s)ds,$$
where $v(t)$ is also locally integrable. As it was observed above by these authors and several others, the problem of compactness of these integral operators is much more complex and trickier than it appears to be. This was what motivated our studies and the discussion we promote with this work.

To be more accurate, in Section 2 we begin by using Interpolation Theory in order to give a shorter proof to Hardy-Littlewood's classical result on fractional integrals, however here in the general context of Bochner-Lebesgue functions. We end this section discussing the range of Riemann-Liouville fractional integral operator.

Section 3 is dedicated to address our main results. We begin introducing Riemann-Liouville and Caputo fractional derivatives, some of their associated classical results, an interpolation theorem and the classical Simon's characterization of compact subsets of $L^p(t_0,t_1;X)$. Then we use those results to prove that Riemann-Liouville fractional integral operator from $L^p(t_0,t_1;X)$ into $L^{q}(t_0,t_1;X)$, with $q< p/(1-p\alpha)$, is a compact operator.

\section{On the Continuity of the Riemann-Liouville Fractional Integral in Bochner-Lebesgue spaces}

We already know that Riemann-Liouville fractional integral of order $\alpha>0$ defines a bounded operator from $L^p(t_0,t_1;X)$ into itself. However, in this section we improve this result. In fact, we want to verify that if $p\in(1,\infty)$ and $\alpha\in(0,1/p)$, then $1/(1-p\alpha)$ is the greatest exponent such that
$$J_{t_0,t}^\alpha:L^p(I;X)\rightarrow L^{1/(1-p\alpha)}(I;X),$$
is well defined and, in fact, defines a bounded operator, regardless of whether $I=[t_0,t_1]$ or $I=[t_0,\infty)$. Here our objective is not only adapt this result to general Bochner-Lebesgue spaces, but also give a more adequate proof.

First, let us fix some concepts (for more details on fractional calculus and the theory of Bochner integrable functions, we may refer to \cite{ArBaHiNe1,CarFe0,Mik1}).

\begin{definition}\label{definitial} If $p\in[1,\infty]$, we represent the set of all Bochner measurable functions $f:I\rightarrow X$, for which $\|f\|_X\in L^{p}(I;\mathbb{R})$, by the symbol $L^{p}(I;{X})$. Moreover, $L^{p}(I;{X})$ is a Banach space when considered with the norm
$$\|f\|_{L^p(I;X)}:=\left\{\begin{array}{ll}\bigg[\displaystyle\int_I{\|f(s)\|^p_X}\,ds\bigg]^{1/p},&\textrm{ if }p\in[1,\infty),\vspace*{0.3cm}\\
\esssup_{s\in I}\|f(s)\|_X,&\textrm{ if }p=\infty.\end{array}\right.$$
When $I=[t_0,t_1]$  (or $I=[t_0,\infty)$) almost everywhere, we shall use the notation $L^p(t_0,t_1;X)$ (or $L^p(t_0,\infty;X)$) besides $L^p(I;X)$. \vspace*{0.2cm}
\end{definition}

With the Bochner-Lebesgue spaces introduced above, we may present the classical Riemann-Liouville fractional integral.

\begin{definition} Assume that $\alpha\in(0,\infty)$, $I=[t_0,t_1]$ or $I=[t_0,\infty)$ and consider function $f:I\rightarrow{X}$. The Riemann-Liouville (RL for short) fractional integral of order $\alpha$ at $t_0$ of function $f$ is defined by
\begin{equation}\label{fracinit} J_{t_0,t}^\alpha f(t):=\dfrac{1}{\Gamma(\alpha)}\displaystyle\int_{t_0}^{t}{(t-s)^{\alpha-1}f(s)}\,ds,\end{equation}
for every $t\in I$ such that integral \eqref{fracinit} exists. Above $\Gamma$ denotes the classical Euler's gamma function.
\end{definition}

\begin{remark}\label{remark5}
%
%
%
For $\alpha>0$, consider function $g_\alpha:\mathbb{R}\rightarrow\mathbb{R}$ given by
\begin{equation*}g_\alpha(r):=\left\{\begin{array}{cl} r^{\alpha-1}/\Gamma(\alpha),&r>0,\vspace{0.2cm}\\
0,&r\leq0.\end{array}\right.\end{equation*}
If $f\in L^1(t_0,t_1;X)$, by defining $f_{t_0}:\mathbb{R}\rightarrow X$ as equal to $f$ in $[t_0,t_1]$ and equal to zero otherwise, we deduce that
$$J_{t_0,t}^\alpha f(t)=\big[\,g_{\alpha}*{f_{t_0}}\,\big](t),\quad \textrm{for almost every }t\in[t_0,t_1],$$
where the symbol $*$ stands for the convolution. We omit the subscript $t_0$ in the above notation when it does not lead to any confusion. For more details on convolutions of Bochner integrable functions see \cite[Chapter XIV]{Mik1}.
%
%
\end{remark}

\subsection{Critical Case}

In order to prove the continuity of the RL fractional integral from $L^p(I;X)$ into $L^{p/(1-p\alpha)}(I;X)$, we find useful to begin this subsection by recalling some theoretical framework about the weak $L^p$ spaces.

\begin{definition}\label{weak1} Consider $f:I\rightarrow\mathbb{R}$ a Lebesgue measurable function. We define the distribution function of $f(t)$, i.e., $\lambda_f:(0,\infty)\rightarrow[0,\infty]$, by
$$\lambda_f(r)=\mu\{t\in I:|f(t)|>r\}.$$
where $\mu$ denotes the Lebesgue measure. Now, if $p\in[1,\infty)$ we define the weak $L^p(I;\mathbb{R})$, or just $L_w^p(I;\mathbb{R})$, to be the set of all Lebesgue measurable functions $f:I\rightarrow\mathbb{R}$ (where functions which agree almost everywhere are identified) such that
\begin{equation*} [f]_{L_w^p(I;\mathbb{R})}:={\Big\{\sup_{r>0}\big[r^p\lambda_f(r)\big]\Big\}^{1/p}}<\infty.\end{equation*}

\end{definition}

\begin{remark}\label{weak2} It is not our intention to discuss the $L_w^p(I;\mathbb{R})$ spaces, for  $p\in[1,\infty)$, as a whole theory in this paper, since such discussion can be found, for instance, in \cite[Chapter 1]{Gra1}. Here, we just note that function $[\,.\,]_{L_w^p(I;\mathbb{R})}:L^p_w(I;\mathbb{R})\rightarrow[0,\infty)$ defines a quasi-norm in $L^p_w(I;\mathbb{R})$, since instead of the triangular inequality it holds that
$$[f+g]_{L_w^p(I;\mathbb{R})}\leq 2^{1/p}\Big([f]_{L^p_w(I;\mathbb{R})}+[g]_{L^p_w(I;\mathbb{R})}\Big).\vspace*{0.2cm}$$
\end{remark}

Below we introduce a standard characterization of some linear operators from $L^p(I;\mathbb{R})$ into $L_w^q(I;\mathbb{R})$; more precisely, the weak-type $(p,q)_I$.
\begin{definition} Suppose that $p\in[1,\infty]$ and $q\in[1,\infty)$. We say that a linear operator $T:L^p(I;\mathbb{R})\rightarrow L_w^q(I;\mathbb{R})$ is of weak-type $(p,q)_I$ if there exists $c=c(p,q)$ such that
$$ [Tf]_{L_w^q(I;\mathbb{R})}\leq c\|f\|_{L^p(I;\mathbb{R})},\quad \forall f\in L^p(I;\mathbb{R}).$$
Also, we say that a linear operator $T:L^p(I;\mathbb{R})\rightarrow L^\infty(I;\mathbb{R})$ is of weak-type $(p,\infty)_I$ if there exists $c=c(p,\infty)$ such that
$$ \|Tf\|_{L^\infty(I;\mathbb{R})}\leq c\|f\|_{L^p(I;\mathbb{R})},\quad \forall f\in L^p(I;\mathbb{R}).$$

In either case, we define the weak-type $(p,q)_I$ norm of the operator $T$ by
$$\|T\|_{(p,q)_I}:=\left\{\begin{array}{ll}\sup\left\{\dfrac{[Tf]_{L_w^q(I;\mathbb{R})}}{\|f\|_{L^p(I;\mathbb{R})}}:{f\in L^p(I;\mathbb{R})\setminus\{0\}}\right\},&\textrm{if }1\leq q<\infty,\vspace*{0.5cm}\\
\sup\left\{\dfrac{\|Tf\|_{L^\infty(I;\mathbb{R})}}{\|f\|_{L^p(I;\mathbb{R})}}:{f\in L^p(I;\mathbb{R})\setminus\{0\}}\right\},&\textrm{if }q=\infty.
\end{array}\right.$$
\end{definition}

\begin{remark} In general, the above notion is formulated in terms of sublinear operators. However, since this work just address linear operators, we avoid this generalization.
\end{remark}

The theory introduced above allow us to present the classical Chebyshev's inequality (see \cite[Theorem 1.1.4]{Gra1} for details on the proof).

\begin{theorem}[Chebyshev's Inequality]\label{chebychev} Let $p\in[1,\infty)$ and $f\in L^p(I;\mathbb{R})$. Then
$$\lambda_f(r)\leq r^{-p}\int_{I}{|f(s)|^p}\,ds,\qquad\forall r>0.$$

\end{theorem}

\begin{remark}\label{inclusoeslp} (i) As a consequence of Chebychev's Inequality, if $p\in[1,\infty)$, it holds that $L^p(I;\mathbb{R})$ is a proper vectorial subspace of $L^p_w(I;\mathbb{R})$ and
$$[f]_{L^p_w(I;\mathbb{R})}\leq\|f\|_{L^p(I;\mathbb{R})},$$
for every $f\in L^p(I;\mathbb{R})$.

$(ii)$ Another three important relations between the $L^p$ spaces and the weak-$L^p$ spaces are the following: 
\begin{itemize}
\item[(a)] If $1\leq p< q<\infty$, it holds that $L_w^q(t_0,t_1;\mathbb{R})\subset L^p(t_0,t_1;\mathbb{R})$. Moreover, we have that
$$\|f\|_{L^p(t_0,t_1;\mathbb{R})}\leq\left(\dfrac{q}{q-p}\right)^{1/p}(t_1-t_0)^{(q-p)/(pq)}[f]_{L^q_w(t_0,t_1;\mathbb{R})},$$
for every $f\in L^q_w(t_0,t_1;\mathbb{R})$. See \cite[Chapter 1]{Gra1} for details.\vspace*{0.2cm}
\item[(b)] If $1\leq p< q\leq\infty$, it holds that $L^q(t_0,t_1;\mathbb{R})\subset L^p_w(t_0,t_1;\mathbb{R})$. Moreover, we have that
$$[f]_{L^p_w(t_0,t_1;\mathbb{R})}\leq\left\{\begin{array}{ll}(t_1-t_0)^{(q-p)/(pq)}\|f\|_{L^q(t_0,t_1;\mathbb{R})},&\textrm{if }q<\infty,\vspace*{0.2cm}\\
(t_1-t_0)^{1/p}\|f\|_{L^q(t_0,t_1;\mathbb{R})},&\textrm{if }q=\infty.\end{array}\right.$$
for every $f\in L^q(t_0,t_1;\mathbb{R})$. This last inequality follows from item $(i)$ of Remark \ref{inclusoeslp} and H\"{o}lder's inequality. \vspace*{0.2cm}
\item[(c)] If $1\leq p< q<\infty$, it holds that $L_w^q(t_0,t_1;\mathbb{R})\subset L^p_w(t_0,t_1;\mathbb{R})$. Moreover, we have that
$$[f]_{L^p_w(t_0,t_1;\mathbb{R})}\leq\left(\dfrac{q}{q-p}\right)^{1/p}(t_1-t_0)^{(q-p)/(pq)}[f]_{L^q_w(t_0,t_1;\mathbb{R})},$$
for every $f\in L_w^q(t_0,t_1;\mathbb{R})$. This follows from subitem $(a)$ and item $(i)$ of Remark \ref{inclusoeslp}.
\end{itemize}
\end{remark}

Now we present a core result that is used in the proof of our version of Theorem \ref{theoHL}, when adapted to Bochner-Lebesgue spaces. It worths to emphasize that the ideas used to prove this next result were inspired by Interpolation Theory (see \cite[Theorem 4.18]{BeSh1} for details).

\begin{theorem}\label{theoHLBochlemma}
Consider $p\in[1,\infty)$, $\alpha\in(0,1/p)$ and assume that $I=[t_0,t_1]$ or $I=[t_0,\infty)$. Then RL fractional integral of order $\alpha$ defines an operator of weak-type $\big(p,{p}/{(1-p\alpha)}\big)_I$. More specifically, it holds that
$$[J^\alpha_{t_0,t}f]_{L_w^{p/(1-p\alpha)}(I;\mathbb{R})}\leq K_{\alpha,p}\|f\|_{L^{p}(I;\mathbb{R})},$$
for every $f\in L^{p}(I;\mathbb{R})$, where
\begin{equation}\label{equationauxconst}K_{\alpha,p}=\left\{\begin{array}{ll}\dfrac{2(p-1)^{\alpha(p-1)}}{\alpha^{1-p\alpha}\,\Gamma(\alpha)(1-\alpha p)^{\alpha(p-1)}},&\textrm{ if }p>1,\vspace*{0.3cm}\\
\dfrac{2}{\alpha^{1-\alpha}\,\Gamma(\alpha)},&\textrm{ if }p=1.\end{array}\right.\end{equation}
\end{theorem}
\begin{proof} The Bochner integrability of $J^\alpha_{t_0,t}f(t)$ follows from Theorem \ref{minkowskiseq}. Since for $f\equiv0$ the inequality is trivial, let us assume that $f\in L^p(I;\mathbb{R})\setminus\{0\}$. From  Remark \ref{remark5}, we already know that $J_{t_0,t}^\alpha f(t)=\big[\,g_{\alpha}*{f_{t_0}}\,\big](t)$, for almost every $t\in I.$ Define, for each $\zeta>0$, functions $g_{\alpha,1,\zeta}\,,\,g_{\alpha,\infty,\zeta}:\mathbb{R}\rightarrow\mathbb{R}$, which are given by
$$g_{\alpha,1,\zeta}(t)=\left\{\begin{array}{ll}g_\alpha(t),&\textrm{if }t\in(0,\zeta),\\0,&\textrm{otherwise},\end{array}\right.\qquad\textrm{and}\qquad g_{\alpha,\infty,\zeta}(t)=\left\{\begin{array}{ll}0,&\textrm{if }t\in(0,\zeta),\\g_\alpha(t),&\textrm{otherwise}.\end{array}\right.$$


Thus, since we may write
$$J_{t_0,t}^\alpha f(t)=\big[\,g_{\alpha,1,\zeta}*{f_{t_0}}\,\big](t)+\big[\,g_{\alpha,\infty,\zeta}*{f_{t_0}}\,\big](t),$$
for almost every $t\in I$, we deduce that for each $r>0$
\begin{multline}\label{inclusionweak}\Big\{t\in I:\big|J_{t_0,t}^\alpha f(t)\big|>r\Big\}\subset\Big\{t\in I:\big|[g_{\alpha,1,\zeta}*f_{t_0}](t)\big|>r/2\Big\}\\\bigcup \Big\{t\in I:\big|[g_{\alpha,\infty,\zeta}*f_{t_0}](t)\big|>r/2\Big\}.\end{multline}
\begin{itemize}
\item[(i)] \underline{For the case $p\in(1,\infty)$.}\vspace*{0.4cm}\\
Since it holds that
$$\big|\big[g_{\alpha,\infty,\zeta}*{f_{t_0}}\big](t)\big|\leq\int_{-\infty}^\infty{\big|g_{\alpha,\infty,\zeta}(t-s){f_{t_0}(s)}\big|}\,ds,$$
Young's Inequality for convolutions ensures the estimate
\begin{multline*}\hspace*{1.2cm}\big|\big[g_{\alpha,\infty,\zeta}*{f_{t_0}}\big](t)\big|\leq \big\|g_{\alpha,\infty,\zeta}*{f_{t_0}}\big\|_{L^{\infty}(\mathbb{R};\mathbb{R})}\leq \|g_{\alpha,\infty,\zeta}\|_{L^{p/(p-1)}(\mathbb{R};\mathbb{R})}\|{f_{t_0}}\|_{L^p(\mathbb{R};\mathbb{R})}
\\= \left[\dfrac{\zeta^{\alpha-(1/p)}}{\Gamma(\alpha)\big[(1-\alpha p)/(p-1)\big]^{(p-1)/p}}\right]\|{f}\|_{L^p(I;\mathbb{R})},\end{multline*}
for almost every $t\in I$. Thus, if we choose $\zeta_r>0$ (in fact, there is only one) such that
\begin{equation}\label{espcons}\left[\dfrac{\zeta_r^{\alpha-(1/p)}}{\Gamma(\alpha)\big[(1-\alpha p)/(p-1)\big]^{(p-1)/p}}\right]\|{f}\|_{L^p(I;\mathbb{R})}=r/2,\end{equation}
we would obtain that
\begin{equation}\label{inclusionweak1}\mu\Big(\big\{t\in I:\big|[g_{\alpha,\infty,\zeta_r}*f_{t_0}](t)\big|>r/2\big\}\Big)=0,\end{equation}
where $\mu$ above is the Lebesgue measure. Thus, from \eqref{inclusionweak} we deduce that
\begin{multline}\label{inclusionweak2}\hspace*{1.2cm}\big\{t\in I:\big|J_{t_0,t}^\alpha f(t)\big|>r\big\}\setminus \big\{t\in I:\big|[g_{\alpha,\infty,\zeta_r}*f_{t_0}](t)\big|>r/2\big\}\\\subset\big\{t\in I:\big|[g_{\alpha,1,\zeta_r}*f_{t_0}](t)\big|>r/2\big\},\end{multline}
and consequently from \eqref{inclusionweak1} and \eqref{inclusionweak2} we obtain
\begin{multline}\label{quasefim00}\hspace*{1.2cm}\mu\Big(\big\{t\in I:\big|J_{t_0,t}^\alpha f(t)\big|>r\big\}\Big)\leq\mu\Big(\big\{t\in I:\big|[g_{\alpha,1,\zeta_r}*f_{t_0}](t)\big|>r/2\big\}\Big).\end{multline}

Finally, since Young's inequality to convolutions ensures that
\begin{equation}\label{quasefim01}\hspace*{0.9cm}\|g_{\alpha,1,\zeta}*{f_{t_0}}\|_{L^p(\mathbb{R};\mathbb{R})}\leq \|g_{\alpha,1,\zeta}\|_{L^1(\mathbb{R};\mathbb{R})}\|{f_{t_0}}\|_{L^p(\mathbb{R};\mathbb{R})}=\left[\dfrac{\zeta^\alpha}{\Gamma(\alpha+1)}\right]\|f\|_{L^p(I;\mathbb{R})},\end{equation}
for any $\zeta>0$, by applying Theorem \ref{chebychev} in the right side of \eqref{quasefim00} and by considering \eqref{quasefim01}, we achieve the estimate
$$\hspace*{0.5cm}\mu\Big(\big\{t\in I:\big|J_{t_0,t}^\alpha f(t)\big|>r\big\}\Big)\leq (r/2)^{-p}\left[\dfrac{\zeta_r^\alpha}{\Gamma(\alpha+1)}\right]^p\|f\|^p_{L^p(I;\mathbb{R})},\quad\forall r>0.$$
But then, taking into account \eqref{espcons}, we deduce that
\begin{multline*}\hspace*{1.2cm}\mu\Big(\big\{t\in I:\big|J_{t_0,t}^\alpha f(t)\big|>r\big\}\Big)\\\leq \left[\dfrac{2\,\|f\|_{L^p(I;\mathbb{R})}}{r\alpha^{1-p\alpha}\,\Gamma(\alpha)\big[(1-\alpha p)/(p-1)\big]^{\alpha(p-1)}}\right]^{p/(1-p\alpha)},\quad\forall r>0.\end{multline*}

Hence, we deduce the estimate
$$\hspace*{1.1cm}\left\{\sup_{r>0}\Big[r^{p/(1-p\alpha)}\mu\Big(\big\{t\in I:\big|J_{t_0,t}^\alpha f(t)\big|>r\big\}\Big)\Big]\right\}^{(1-p\alpha)/p}\\\leq K_{\alpha,p}\|f\|_{L^p(I;\mathbb{R})},$$
where
$$K_{\alpha,p}=\dfrac{2(p-1)^{\alpha(p-1)}}{\alpha^{1-p\alpha}\,\Gamma(\alpha)(1-\alpha p)^{\alpha(p-1)}}.$$
\vspace*{0.4cm}

\item[(ii)] \underline{For the case $p=1$.}\vspace*{0.4cm}\\
We follow the ideas presented above. Let us start by noting that
$$\hspace*{0.8cm}|g_{\alpha,\infty,\zeta}*{f_{t_0}}(t)|\leq \|g_{\alpha,\infty,\zeta}\|_{L^{\infty}(\mathbb{R};\mathbb{R})}\|{f_{t_0}}\|_{L^1(\mathbb{R};\mathbb{R})}
= \left[\dfrac{\zeta^{\alpha-1}}{\Gamma(\alpha)}\right]\|{f}\|_{L^1({I};\mathbb{R})},$$
for almost $t\in I$. Hence, if we define $\zeta_r>0$ as a solution of the equality
\begin{equation}\label{aquasefim00121}\left[\dfrac{\zeta_r^{\alpha-1}}{\Gamma(\alpha)}\right]\|{f}\|_{L^1({I};\mathbb{R})}=r/2,\end{equation}
we would obtain that $\big\{t\in I:\big|[g_{\alpha,\infty,\zeta_r}*f_{t_0}](t)\big|>r/2\big\}$ has Lebesgue measure zero. Like before, we deduce that
\begin{equation}\label{aquasefim00}\hspace*{1.1cm}\mu\Big(\big\{t\in I:\big|J_{t_0,t}^\alpha f(t)\big|>r\big\}\Big)\leq\mu\Big(\big\{t\in I:\big|[g_{\alpha,1,\zeta_r}*f_{t_0}](t)\big|>r/2\big\}\Big).\end{equation}
Finally, since Young's inequality to convolutions ensures that
\begin{equation}\label{aquasefim01}\hspace*{0.9cm}\|g_{\alpha,1,\zeta}*{f_{t_0}}\|_{L^1(\mathbb{R};\mathbb{R})}\leq \|g_{\alpha,1,\zeta}\|_{L^1(\mathbb{R};\mathbb{R})}\|{f_{t_0}}\|_{L^1(\mathbb{R};\mathbb{R})}=\left[\dfrac{\zeta^\alpha}{\Gamma(\alpha+1)}\right]\|f\|_{L^1({I};\mathbb{R})},\end{equation}
for any $\zeta>0$, by applying Theorem \ref{chebychev} in the right side of \eqref{aquasefim00} and by considering \eqref{aquasefim01} and \eqref{aquasefim00121}, we achieve the estimate
\begin{multline*}\hspace*{1.2cm} \mu\Big(\big\{t\in I:\big|J_{t_0,t}^\alpha f(t)\big|>r\big\}\Big)\leq (r/2)^{-1}\left[\dfrac{\zeta_r^\alpha}{\Gamma(\alpha+1)}\right]\|f\|_{L^1({I};\mathbb{R})}\\ \leq \left[\dfrac{2\|f\|_{L^1({I};\mathbb{R})}}{r\alpha^{1-\alpha}\Gamma(\alpha)}\right]^{1/(1-\alpha)}.\end{multline*}

Therefore
$$\hspace*{0.8cm}\left\{\sup_{r>0}\Big[r^{1/(1-\alpha)}\mu\Big(\big\{t\in I:\big|J_{t_0,t}^\alpha f(t)\big|>r\big\}\Big)\Big]\right\}^{1-\alpha}\\\leq \left[\dfrac{2}{\alpha^{1-\alpha}\Gamma(\alpha)}\right]\|f\|_{L^1({I};\mathbb{R})}.$$
\end{itemize}
\end{proof}

\begin{remark}\label{theoHLBochlemmarameark} A consequence of Theorem \ref{theoHLBochlemma} is that
$$\|J_{t_0,t}^\alpha\|_{(p,p/(1-p\alpha))_I}\leq K_{\alpha,p},$$
for every $p\in[1,\infty)$ and $\alpha\in(0,1/p)$, whether $I=[t_0,t_1]$ or $I=[t_0,\infty)$. Recall that $K_{\alpha,p}$ is given in \eqref{equationauxconst}.
\end{remark}

\begin{remark}
  For each $p\in[1,\infty)$, we may define the vectorial spaces $L^p_w(I;X)$ (which are quasi-normed vectorial spaces), by following the same ideas used to define the vectorial spaces $L^p_w(I;\mathbb{R})$; see Definition \ref{weak1} and Remark \ref{weak2}.\vspace*{0.2cm}
\begin{itemize}
\item[(a)] As a consequence of Theorem \ref{chebychev}, it holds that $L^p({I};X)\subset L^p_w({I};X)$ and
$$[f]_{L^p_w({I};X)}\leq\|f\|_{L^p({I};X)},$$
for each $f\in L^p({I};X)$.\vspace*{0.2cm}
\item[(b)] If $p\in[1,\infty)$ and $\alpha\in(0,1/p)$, a consequence of Theorem \ref{theoHLBochlemma} is that
$$[J^\alpha_{t_0,t}f]_{L_w^{p/(1-p\alpha)}(I;X)}\leq K_{\alpha,p}\|f\|_{L^{p}(I;X)},$$
for every $f\in L^{p}(I;X)$, where $K_{\alpha,p}$ is given in \eqref{equationauxconst}.  In fact, just observe that for any $f\in L^{p}(I;X)$, we have by Theorem \ref{theoHLBochlemma} that
\begin{multline*}\hspace*{2cm}[J^\alpha_{t_0,t}f]_{L_w^{p/(1-p\alpha)}(I;X)}=[J^\alpha_{t_0,t}\|f\|_X]_{L_w^{p/(1-p\alpha)}(I;\mathbb{R})}
{\leq}\\ K_{\alpha,p}\|\|f\|_X\|_{L^{p}(I;\mathbb{R})}=K_{\alpha,p}\|f\|_{L^{p}(I;X)}.\end{multline*}

The above argument allows us to conclude that RL fractional integral of order $\alpha$ defines a ``bounded'' operator from $L^p({I};X)$ into $L_w^{p/(1-p\alpha)}({I};X)$. The term bounded is placed between quotation marks because the quasi-norm of the $L_w^{p/(1-p\alpha)}({I};X)$ does not make it a Banach space.\vspace*{0.2cm}
\end{itemize}
\end{remark}

Let us now recall Marcinkiewicz's interpolation theorem (for details on the classical proof of this result see \cite[Theorem 1]{Zig01}).

\begin{theorem}[Marcinkiewicz's Interpolation]\label{marcinkiewicz} Consider real numbers $p_1,p_2,q_1,q_2$ such that $1\leq p_i\leq q_i\leq\infty,$ for $i=1,2,$ and $q_1\neq q_2$. Suppose that $T$ is a linear operator that is simultaneously of weak type $(p_1,q_1)_I$ and $(p_2,q_2)_I$. If $0<\theta<1$ and the real numbers $p_\theta,q_\theta$ are such that
$$\dfrac{1}{p_\theta}=\dfrac{1-\theta}{p_1}+\dfrac{\theta}{p_2}\quad\textrm{and}\quad\dfrac{1}{q_\theta}=\dfrac{1-\theta}{q_1}+\dfrac{\theta}{q_2},$$
then there exists $M_\theta>0$ such that
$$\|Tf\|_{L^{q_\theta}(I;\mathbb{R})}\leq M_\theta\|f\|_{L^{p_\theta}(I;\mathbb{R})},$$
for every $f\in L^{p_\theta}(I;\mathbb{R})$.
\end{theorem}

\begin{remark}\label{minkuconst} It is possible to compute the value $M_\theta$ introduced in the above theorem. More precisely, if we recall that $\|T\|_{(p_i,q_i)_I}$ denotes the weak-type $(p_i,q_i)_I$ norm of the operator $T$, for $i=1,2$, then the proof of Theorem \ref{marcinkiewicz} gives us that (see identity (3.12) in \cite{Zig01} for details)
$$M_\theta=K \|T\|_{(p_1,q_1)_I}^{1-\theta} \|T\|_{(p_2,q_2)_I}^{\theta},$$
where
$$K=2(q_\theta)^{1/q_\theta}\left[\dfrac{(p_1/p_\theta)^{q_1/p_1}}{q_\theta-q_1}+\dfrac{(p_2/p_\theta)^{q_2/p_2}}{q_2-q_\theta}\right]^{1/q_\theta}.$$

It is important to note that $M_\theta$ may possibly be bigger than $\|T\|_{\mathcal{L}(L^{p_\theta}(I;\mathbb{R}),L^{q_\theta}(I;\mathbb{R}))}$.
\end{remark}

Following from Theorem \ref{theoHLBochlemma} and Theorem \ref{marcinkiewicz}, we now present our version of Theorem \ref{theoHL} to Bochner-Lebesgue spaces. Note that it is a shorter proof using interpolation theory of the already mentioned result o Hard-Littlewood applied to RL fractional integrals.

\begin{theorem}\label{theoHLBoch}Consider $p\in(1,\infty)$, $\alpha\in(0,1/p)$ and assume that $I=[t_0,t_1]$ or $I=[t_0,\infty)$. If $f\in L^p(I;X)$, then $J_{t_0,t}^\alpha f(t)$ is Bochner integrable in $I$ and belongs to $L^{p/(1-p\alpha)}(I;X)$. Moreover, there exists $C_{\alpha,p}>0$ such that
\begin{equation*}\left[\int_{I}{\left\|J^\alpha_{t_0,s}f(s)\right\|_X^{p/(1-p\alpha)}}\,ds\right]^{(1-p\alpha)/p}\leq C_{\alpha,p}\left[\int_{I}{\|f(s)\|_X^p}\,ds\right]^{1/p}.\end{equation*}
In other words, $J_{t_0,t}^\alpha$ defines a bounded operator from $L^p(I;X)$ into $L^{p/(1-p\alpha)}(I;X)$.

\end{theorem}
\begin{proof} The Bochner integrability of $J^\alpha_{t_0,t}f(t)$ follows from Theorem \ref{minkowskiseq}. Since $p\in (1,\infty)$ and $\alpha\in(0,1/p)$, choose $p_1\in(1,p)$ and $p_2\in(p,1/\alpha)$. Define $q_1=p_1/(1-p_1\alpha)$, $q_2=p_2/(1-p_2\alpha)$, and observe that Theorem \ref{theoHLBochlemma} ensures that $J_{t_0,t}^\alpha$ defines an operator of weak-type $(p_1,q_1)_I$ and $(p_2,q_2)_I$.

Hence, for each $\theta\in(0,1)$, Theorem \ref{marcinkiewicz} ensures the existence of $M_\theta:=C_{\alpha,p}>0$, that depends of the variables $p_1,p_2,p$ and $\alpha$, such that
$$\|J_{t_0,t}^\alpha f\|_{L^{q_\theta}(I;\mathbb{R})}\leq C_{\alpha,p}\|f\|_{L^{p_\theta}(I;\mathbb{R})},$$
for every $f\in L^{p_\theta}(I;\mathbb{R})$, where
$$\dfrac{1}{p_\theta}=\dfrac{1-\theta}{p_1}+\dfrac{\theta}{p_2}\quad\textrm{and}\quad\dfrac{1}{q_\theta}=\dfrac{1-\theta}{q_1}+\dfrac{\theta}{q_2}.$$

By choosing $\theta=p_2(p-p_1)/[p(p_2-p_1)]$, we deduce that $p_\theta=p$ and $q_\theta=p/(1-p\alpha)$ what lead us to
\begin{equation}\label{auxcont01}\|J_{t_0,t}^\alpha f\|_{L^{p/(1-p\alpha)}(I;\mathbb{R})}\leq C_{\alpha,p}\|f\|_{L^{p}(I;\mathbb{R})},\end{equation}
for every $f\in L^{p}(I;\mathbb{R})$.

Now, since for any $f\in L^{p}(I;X)$ it holds that $\|f\|_X\in L^{p}(I;\mathbb{R})$, by \ref{auxcont01} we deduce that
\begin{equation*}\|J_{t_0,t}^\alpha f\|_{L^{p/(1-p\alpha)}(I;X)}\leq \|J_{t_0,t}^\alpha \|f\|_{X}\|_{L^{{p/(1-p\alpha)}}(I;\mathbb{R})}{\leq} C_{\alpha,p}\|\|f\|_{X}\|_{L^{{p}}(I;\mathbb{R})}=C_{\alpha,p}\|f\|_{L^{p}(I;X)},\end{equation*}
what completes the proof of our theorem.
\end{proof}

\begin{remark} As pointed by Theorem \ref{theoHL}, we emphasize that constant $C_{\alpha,p}$ of Theorem \ref{theoHLBoch} does not depends on the length of $I$.
\end{remark}

\subsection{Remaining Cases} In the beginning of this section we stated that $p/(1-p\alpha)$ is the biggest exponent such that the operator induced by RL fractional integral can define a linear (and also bounded) operator, regardless of whether $I$ is equal to $[t_0,t_1]$ or $[t_0,\infty)$. Let us now prove this fact.

\begin{theorem}\label{maxbigg} Let $p\in(1,\infty)$ and $\alpha\in(0,1/p)$.\vspace*{0.2cm}
\begin{itemize}
\item[(i)] If $q\in\big[1,p/(1-p\alpha)\big]$, then
$$J_{t_0,t}^\alpha:L^p(t_0,t_1;X)\rightarrow L^{q}(t_0,t_1;X),$$
is a bounded operator and for every $f\in L^p(t_0,t_1;X)$ it holds that
\begin{equation}\label{inenov45}\left\|J^\alpha_{t_0,t}\,\,f\right\|_{L^q(t_0,t_1;X)}\leq C_{\alpha,p}(t_1-t_0)^{[(p-q)/pq]+\alpha}\|f\|_{L^p(t_0,t_1;X)},\end{equation}
where $C_{\alpha,p}$ is given in Theorem \ref{theoHLBoch}. Moreover, for any $\eta\in(p/(1-p\alpha),\infty]$, there exists $f_\eta\in L^p(t_0,t_1;X)$ such that $J_{t_0,t}^\alpha f_\eta(t)$ does not belongs to $L^\eta(t_0,t_1;X)$. \vspace*{0.2cm}
\item[(ii)] If $\eta\in[1,\infty]$, however $\eta\not=p/(1-p\alpha)$, there exists $f_\eta\in L^p(t_0,\infty;X)$ such that $J_{t_0,t}^\alpha f_\eta(t)$ does not belongs to $L^\eta(t_0,\infty;X)$.
\end{itemize}
\end{theorem}

\begin{proof} $(i)$ Let $q\in\big[1,p/(1-p\alpha)\big)$ and consider the bounded (embedding) operator
$$\begin{array}{lclc}E:&L^{p/(1-p\alpha)}(t_0,t_1;X)&\rightarrow& L^q(t_0,t_1;X),\vspace*{0.0cm}\\
&f(t)&\mapsto& f(t),\end{array}$$
which satisfies, for any $f\in L^{p/(1-p\alpha)}(t_0,t_1;X)$, the estimate
$$\|Ef\|_{L^{q}(t_0,t_1;X)}\leq(t_1-t_0)^{[(p-q)/pq]+\alpha}\|f\|_{L^{p/(1-p\alpha)}(t_0,t_1;X)}.$$

Since
$$J_{t_0,t}^\alpha:L^p(t_0,t_1;X)\rightarrow L^{q}(t_0,t_1;X)$$
is given by the composition of the bounded operator (see Theorem \ref{theoHLBoch})
$$J_{t_0,t}^\alpha:L^p(t_0,t_1;X)\rightarrow L^{p/(1-p\alpha)}(t_0,t_1;X),$$
with the bounded operator $E$, we deduce that it is bounded and that it satisfies \eqref{inenov45}.

To verify the second part of item $(i)$, we consider:
\begin{itemize}
\item[(a)] for $\eta\in(p/(1-p\alpha),\infty)$, choose ${\beta_\eta}\in\big(\alpha+(1/\eta),1/p\big)$;\vspace*{0.2cm}
\item[(b)] for $\eta=\infty$, choose ${\beta_\eta}\in\big(\alpha,1/p\big)$.
\end{itemize}
Then, fix $x\in X$, with $\|x\|_X=1$, and let function $f_\eta:(t_0,t_1]\rightarrow X$ be given by $f_\eta(t)=(t-t_0)^{-{\beta_\eta}}x$.

Note that $f_\eta$ belongs to $L^p(t_0,t_1;X)$ and
$$J_{t_0,t}^\alpha f_\eta(t)=\dfrac{\Gamma(1-{\beta_\eta})}{\Gamma(1+\alpha-{\beta_\eta})}(t-t_0)^{\alpha-{\beta_\eta}}x,$$
does not belongs to $L^\eta(t_0,t_1;X)$, since $(\alpha-{\beta_\eta})\eta+1<0$ (in the case $(a)$) and $\alpha-{\beta_\eta}<0$ (in the case $(b)$).

$(ii)$ If $\eta\in[1,p/(1-p\alpha))$, consider ${\beta_\eta}\in\big(1/p,\min{\{1,\alpha+(1/\eta)\}}\big)$, $x\in X$ and define function $f_\eta:[t_0,\infty)\rightarrow X$ by
$$f_\eta(t)=\left\{\begin{array}{ll}0,&\textrm{if }t\in[t_0,t_0+1],\vspace*{0.2cm}\\
(t-t_0)^{-{\beta_\eta}}x,&\textrm{if }t\in(t_0+1,\infty).\end{array}\right.$$
This set up allows us to conclude that $f_\eta\in L^p(t_0,\infty;X)$. On the other hand, observe that
$$J_{t_0,t}^\alpha f_\eta(t)=\left\{\begin{array}{ll}0,&\textrm{if }t\in[t_0,t_0+1],\vspace*{0.2cm}\\
\left[\dfrac{(t-t_0)^{\alpha-{\beta_\eta}}}{\Gamma(\alpha)}\right]\left[\displaystyle\int_{1/(t-t_0)}^1{(1-s)^{\alpha-1}s^{-{\beta_\eta}}}\,ds\right]\,x,&\textrm{if }t\in(t_0+1,\infty).\end{array}\right.$$
Hence, we have that
$$\|J_{t_0,t}^\alpha f_\eta\|^\eta_{L^{\eta}(t_0,\infty;X)}=
\int_{t_0+1}^\infty\left[\dfrac{(t-t_0)^{\alpha-{\beta_\eta}}}{\Gamma(\alpha)}\right]^{\eta}\left[\displaystyle\int_{1/(t-t_0)}^1{(1-s)^{\alpha-1}s^{-{\beta_\eta}}}\,ds\right]^{\eta}\,dt.$$
However, since
$$\int_{1/(t-t_0)}^1{(1-s)^{\alpha-1}s^{-{\beta_\eta}}}\,ds\geq\int_{1/(t-t_0)}^1{(1-s)^{\alpha-1}}\,ds= \alpha^{-1}\left(\dfrac{t-t_0-1}{t-t_0}\right)^{\alpha},$$
for almost every $t\in[t_0+1,\infty)$, we obtain
$$\|J_{t_0,t}^\alpha f_\eta\|^\eta_{L^{\eta}(t_0,\infty;X)}\geq
\left\{\int_{t_0+1}^\infty\left[\dfrac{(t-t_0-1)^{\alpha}(t-t_0)^{-{\beta_\eta}}}{\Gamma(\alpha+1)}\right]^{\eta}\,dt\right\}.$$

By changing the variable $t=w+t_0$, we obtain
$$\|J_{t_0,t}^\alpha f_\eta\|^\eta_{L^{\eta}(t_0,\infty;X)}\geq
\left(\dfrac{1}{{\Gamma(\alpha+1)}}\right)^\eta\left\{\int_{1}^\infty{(w-1)^{\alpha\eta}w^{-{\beta_\eta}\eta}}\,dw\right\}.$$

Finally, since $-{\beta_\eta}\eta<0$, we have
\begin{multline*}\|J_{t_0,t}^\alpha f_\eta\|^\eta_{L^{\eta}(t_0,\infty;X)}\geq
\left(\dfrac{1}{{\Gamma(\alpha+1)}}\right)^\eta\left\{\int_{1}^r{(w-1)^{\alpha\eta}w^{-{\beta_\eta}\eta}}\,dw\right\}\\
\geq\left(\dfrac{1}{{\Gamma(\alpha+1)}}\right)^\eta\dfrac{(r-1)^{\alpha\eta+1}r^{-{\beta_\eta}\eta}}{{\big(\alpha\eta+1\big)}},\end{multline*}
for any $1< r <\infty$. Therefore
 $$ \|J_{t_0,t}^\alpha f_\eta\|^\eta_{L^{\eta}(t_0,\infty;X)}
\geq\lim_{r\rightarrow\infty}\left(\dfrac{1}{{\Gamma(\alpha+1)}}\right)^\eta\dfrac{(r-1)^{\alpha\eta+1}r^{-{\beta_\eta}\eta}}{{\big(\alpha\eta+1\big)}}=\infty,
$$
i.e., $J_{t_0,t}^\alpha f_\eta(t)$ does not belongs to $L^\eta(t_0,t_1;X)$.

Now assume that $\eta\in(p/(1-p\alpha),\infty]$. In this case we can proceed almost like it was done in item $(i)$. To be more specific, we consider:
\begin{itemize}
\item[(a)] for $\eta\in(p/(1-p\alpha),\infty)$, choose ${\beta_\eta}\in\big(\alpha+(1/\eta),1/p\big)$;\vspace*{0.2cm}
\item[(b)] for $\eta=\infty$, choose ${\beta_\eta}\in\big(\alpha,1/p\big)$.
\end{itemize}
Then we fix $x\in X$, with $\|x\|_X=1$, and define function $f_\eta:(t_0,\infty)\rightarrow X$ by
$$f_\eta(t)=\left\{\begin{array}{ll}(t-t_0)^{-{\beta_\eta}}x,&\textrm{if }t\in[t_0,t_0+1],\vspace*{0.2cm}\\
0,&\textrm{if }t\in(t_0+1,\infty).\end{array}\right.$$

Like it was done before, we conclude that $f_\eta\in L^p(t_0,\infty;X)$. However, since
$$J_{t_0,t}^\alpha f_\eta(t)=\left\{\begin{array}{ll}\dfrac{\Gamma(1-{\beta_\eta})}{\Gamma(1+\alpha-{\beta_\eta})}(t-t_0)^{\alpha-{\beta_\eta}}x,&\textrm{if }t\in[t_0,t_0+1],\vspace*{0.3cm}\\
\dfrac{x}{\Gamma(\alpha)}\displaystyle\int_{t_0}^{t_0+1}(t-s)^{\alpha-1}(s-t_0)^{-{\beta_\eta}}\,ds,&\textrm{if }t\in(t_0+1,\infty),\end{array}\right.$$
we deduce that
$$\|J_{t_0,t}^\alpha f_\eta\|_{L^\eta(t_0,\infty;X)}\geq\left\|\dfrac{\Gamma(1-{\beta_\eta})}{\Gamma(1+\alpha-{\beta_\eta})}(t-t_0)^{\alpha-{\beta_\eta}}x\right\|_{L^\eta(t_0,t_0+1;X)}=\infty,$$
i.e., $J_{t_0,t}^\alpha f_\eta(t)$ does not belongs to $L^\eta(t_0,\infty;X)$.
\end{proof}

In the beginning of this section we have showed that Theorem \ref{theoHLBochlemma} together with Theorem \ref{marcinkiewicz} could be converted in Theorem \ref{theoHLBoch}. In other words, we have showed that for any $p\in(1,\infty)$ and $\alpha\in(0,1/p)$, RL fractional integral of order $\alpha$ defines a bounded linear operator from $L^p(I;X)$ into $L_w^{p/(1-p\alpha)}(I;X)$. Now, let us consider the critical case $p=1$.

\begin{theorem}\label{critcasep1} Assume that $\alpha\in(0,1)$.
\begin{itemize}
\item[(i)] If $f\in L^1(t_0,t_1;X)$ we have that $J_{t_0,t}^\alpha f\in L^q(t_0,t_1;X)$, for any $q\in\big[1,1/(1-\alpha)\big)$, and it also holds that
$$\hspace*{1cm}\big\|J_{t_0,t}^\alpha f\big\|_{L^q(t_0,t_1;\mathbb{R})}\leq K_{\alpha,1}\left(\dfrac{(t_1-t_0)^{[1-q(1-\alpha)]}}{1-q(1-\alpha)}\right)^{1/q}\|f\|_{L^1(t_0,t_1;\mathbb{R})},$$
where $K_{\alpha,1}$ is given in Theorem \ref{theoHLBochlemma}. In other words, $J_{t_0,t}^\alpha$ defines a bounded operator from $L^1(t_0,t_1;X)$ into $L^q(t_0,t_1;X)$, for any $q\in\big[1,1/(1-\alpha)\big)$.\vspace*{0.2cm}
\item[(ii)] There exists a function $f\in L^1(t_0,t_1;X)$ such that $J^\alpha_{t_0,t}f\not\in L^{1/(1-\alpha)}(t_0,t_1;X)$.\vspace*{0.2cm}
\item[(iii)] For any $\eta\in[1,\infty]$, there exists $f_\eta\in L^1(t_0,\infty;X)$ such that $J^\alpha_{t_0,t}f\not\in L^{\eta}(t_0,\infty;X)$.\vspace*{0.2cm}
\end{itemize}
\end{theorem}

\begin{proof} $(i)$ Note that subitem $(a)$ of Remark \ref{inclusoeslp} ensures that
\begin{multline*}\big\|J_{t_0,t}^\alpha f\big\|_{L^q(t_0,t_1;X)}\leq\big\|J_{t_0,t}^\alpha \|f\|_X\big\|_{L^q(t_0,t_1;\mathbb{R})}
\\\leq\left(\dfrac{1}{1-q(1-\alpha)}\right)^{1/q}(t_1-t_0)^{[1-q(1-\alpha)]/q}\big[J_{t_0,t}^\alpha \|f\|_X\big]_{L_w^{1/(1-\alpha)}(t_0,t_1;\mathbb{R})}.\end{multline*}
Thus, Theorem \ref{theoHLBochlemma} ensures that
$$\big\|J_{t_0,t}^\alpha f\big\|_{L^q(t_0,t_1;X)}\leq K_{\alpha,1}\left(\dfrac{1}{1-q(1-\alpha)}\right)^{1/q}(t_1-t_0)^{[1-q(1-\alpha)]/q}\|f\|_{L^1(t_0,t_1;X)},$$
like we wanted.

$(ii)$ Inspired by \cite[Section 3.5 - item (iii)]{HaLi1}, we choose $\beta\in(1,2-\alpha)$, $x\in X$, with $\|x\|_X=1$, and consider $f:(t_0,t_1]\rightarrow X$ given by
$$f(t)=\{(t-t_0)/[2(t_1-t_0)]\}^{-1}\left\{\ln{\big[2(t_1-t_0)/(t-t_0)\big]}\right\}^{-\beta}x.$$
Observe that $f\in L^{1}\big(t_0,t_1;X\big)$, since
\begin{multline}\label{limitadolim}\|f\|_{L^{1}(t_0,t_1;X)}=\int_{t_0}^{t_1}\{(t-t_0)/[2(t_1-t_0)]\}^{-1}\left\{\ln{\big[2(t_1-t_0)/(t-t_0)\big]}\right\}^{-\beta}\,dt
\\\hspace*{0cm}\stackrel{t=2(t_1-t_0)s+t_0}= 2(t_1-t_0)\int_{0}^{1/2}s^{-1}\left(\log{s^{-1}}\right)^{-\beta}\,ds\\\stackrel{u=\ln{s^{-1}}}=
2(t_1-t_0)\int_{\ln{2}}^{\infty}{u^{-\beta}}\,du=\dfrac{2(t_1-t_0)\big(\log{2}\big)^{1-\beta}}{\beta-1}<\infty.\end{multline}

On the other hand, observe that
\begin{multline*}J_{t_0,t}^\alpha f(t)=\dfrac{x}{\Gamma(\alpha)}\int_{t_0}^t\left[{\dfrac{(t-s)^{\alpha-1}\big[2(t_1-t_0)/(s-t_0)\big]}{\left\{\ln{\big[2(t_1-t_0)/(s-t_0)\big]}\right\}^{\beta}}}\right]\,ds\\ \stackrel{s=2(t_1-t_0)w+t_0}=\dfrac{x\big[2(t_1-t_0)\big]^\alpha}{\Gamma(\alpha)}\int_{0}^{\frac{t-t_0}{2(t_1-t_0)}}{\left(\left[\dfrac{t-t_0}{2(t_1-t_0)}\right]-w\right)^{\alpha-1}w^{-1}\left(\ln{w^{-1}}\right)^{-\beta}}\,dw,\end{multline*}
for almost every $t\in[t_0,t_1]$. Therefore we have
\begin{multline*}\|J_{t_0,t}^\alpha f(t)\|_X\geq\left(\dfrac{\big[2(t_1-t_0)\big]^\alpha}{\Gamma(\alpha)}\right)\left(\dfrac{t-t_0}{2(t_1-t_0)}\right)^{\alpha-1}\int_{0}^{\frac{t-t_0}{2(t_1-t_0)}}w^{-1}\left(\ln{w^{-1}}\right)^{-\beta}\,dw\\
\hspace*{3cm}\stackrel{u=\ln{w^{-1}}}=\left(\dfrac{2(t_1-t_0)(t-t_0)^{\alpha-1}}{\Gamma(\alpha)}\right)\int_{\ln\left[2(t_1-t_0)/(t-t_0)\right]}^{\infty}u^{-\beta}\,du
\\=\dfrac{2(t_1-t_0)(t-t_0)^{\alpha-1}\left\{\ln\left[2(t_1-t_0)/(t-t_0)\right]\right\}^{1-\beta}}{\Gamma(\alpha)(\beta-1)},\end{multline*}
for almost every $t\in[t_0,t_1]$. Thus,
\begin{multline*}\|J_{t_0,t}^\alpha f\|^{1/(1-\alpha)}_{L^{1/(1-\alpha)}(t_0,t_1;X)}\\
\geq\left[\dfrac{2(t_1-t_0)}{\Gamma(\alpha)(\beta-1)}\right]^{1/(1-\alpha)}\int_{t_0}^{t_1}(s-t_0)^{-1}\big\{\ln\left[2(t_1-t_0)/(s-t_0)\right]\big\}^{(1-\beta)/(1-\alpha)}\,ds
\\\stackrel{u=\ln\left[2(t_1-t_0)/(s-t_0)\right]}=\left[\dfrac{2(t_1-t_0)}{\Gamma(\alpha)(\beta-1)}\right]^{1/(1-\alpha)}\int_{\ln{2}}^{\infty}{u^{(1-\beta)/(1-\alpha)}}\,du=\infty.\end{multline*}
In other words, $J_{0,t}^\alpha f(t)$ does not belongs to $L^{1/(1-\alpha)}\big(t_0,t_1;X\big)$.

$(iii)$ At first, assume that $\eta\in[1/(1-\alpha),\infty)$. Let us follow the steps of item $(ii)$. To this end, choose ${\beta_\eta}\in(1,1+(1/\eta))$, $x\in X$, with $\|x\|_X=1$, $t_1>t_0$ and consider $f_\eta:(t_0,\infty)\rightarrow X$ given by
$$f_\eta(t)=\left\{\begin{array}{ll}\{(t-t_0)/[2(t_1-t_0)]\}^{-1}\left\{\log{\big[2(t_1-t_0)/(t-t_0)\big]}\right\}^{-{\beta_\eta}}x,&\textrm{if }t\in(t_0,t_1],\vspace*{0.2cm}\\
0,&\textrm{if }t\in(t_1,\infty).\end{array}\right.$$

Note that $f_\eta\in L^{1}\big(t_0,\infty;X\big)$, since $1+(1/\eta)<2-\alpha$ and inequality \eqref{limitadolim} ensures that
$$\|f_\eta\|_{L^{1}(t_0,\infty;X)}=\|f_\eta\|_{L^{1}(t_0,t_1;X)}< \infty.$$
On the other hand, observe that
\begin{multline*}\|J_{t_0,t}^\alpha f_\eta\|^{\eta}_{L^{\eta}(t_0,\infty;X)}\geq \|J_{t_0,t}^\alpha f_\eta\|^{\eta}_{L^{\eta}(t_0,t_1;X)}\\
\geq\left[\dfrac{2(t_1-t_0)}{\Gamma(\alpha)({\beta_\eta}-1)}\right]^{\eta}\int_{t_0}^{t_1}(s-t_0)^{(\alpha-1)\eta}\big\{\log\left[2(t_1-t_0)/(s-t_0)\right]\big\}^{(1-{\beta_\eta})\eta}\,ds
\\\stackrel{u=\log\left[2(t_1-t_0)/(s-t_0)\right]}=\left\{\dfrac{[2(t_1-t_0)]^{\alpha\eta+1}}{[\Gamma(\alpha)({\beta_\eta}-1)]^\eta}\right\}\int_{\log{2}}^{\infty}{e^{[(1-\alpha)\eta-1]u}u^{(1-{\beta_\eta})\eta}}\,du=\infty.\end{multline*}

If $\eta=\infty$, choose ${\beta_\eta}\in(\alpha,1)$, $x\in X$, with $\|x\|_X=1$, $t_1>t_0$ and $f_\eta:(t_0,\infty)\rightarrow X$ given by
$$f_\eta(t)=\left\{\begin{array}{ll}(t-t_0)^{-{\beta_\eta}}x,&\textrm{if }t\in(t_0,t_1],\vspace*{0.2cm}\\
0,&\textrm{if }t\in(t_1,\infty).\end{array}\right.$$
Then $f_\eta\in L^1(t_0,\infty;X)$, however
$$J_{t_0,t}^\alpha f_{\eta}(t)=\left\{\begin{array}{ll}\dfrac{\Gamma(1-{\beta_\eta})}{\Gamma(1+\alpha-{\beta_\eta})}(t-t_0)^{\alpha-{\beta_\eta}}x,&\textrm{if }t\in[t_0,t_1],\vspace*{0.3cm}\\
\dfrac{x}{\Gamma(\alpha)}\displaystyle\int_{t_0}^{t_0+1}(t-s)^{\alpha-1}(s-t_0)^{-{\beta_\eta}}\,ds,&\textrm{if }t\in(t_1,\infty),\end{array}\right.$$
does not belongs to $L^\infty(t_0,\infty;X)$.

Finally, if $\eta\in[1,1/(1-\alpha))$, consider ${\beta_\eta}\in\big(1,\alpha+(1/\eta)\big)$, $x\in X$ and define function $f_\eta:[t_0,\infty)\rightarrow X$ by
$$f_\eta(t)=\left\{\begin{array}{ll}0,&\textrm{if }t\in[t_0,t_1],\vspace*{0.2cm}\\
(t-t_0)^{-{\beta_\eta}}x,&\textrm{if }t\in(t_0+1,\infty).\end{array}\right.$$
By the same arguments in the proof of item $(ii)$ of Theorem \ref{maxbigg}, we have $f_\eta\in L^1(t_0,\infty;X)$ and $J_{t_0,t}^\alpha f_\eta(t)\notin L^\eta(t_0,\infty;X)$.

\end{proof}

\section{Compactness of RL Fractional Integral}
\label{compactnessofJ}

In this section we address the compactness of the RL fractional integral. More precisely, we consider $p>1$, $\alpha\in(0,1/p)$ and $q\in\big[1,p/(1-p\alpha)\big)$ in order to prove that the bounded operator
$$J_{t_0,t}^\alpha:L^p(t_0,t_1;X)\rightarrow L^{q}(t_0,t_1;X)$$
is compact. It worths to emphasize that:
\begin{itemize}
\item[(i)] When $q=p/(1-p\alpha)$ the compactness fails to hold. In fact, if we consider $v\in X$ and the sequence of functions $\{f_j(t)\}_{j=1}^\infty\subset L^p(t_0,t_1;X)$ given by
    $$f_j(t)=\left\{\begin{array}{ll}j^{1/p}v,&\textrm{if }t\in\big[t_0,t_0+[(t_1-t_0)/j]\big],\vspace*{0.2cm}\\
    0,&\textrm{if }t\in\big(t_0+[(t_1-t_0)/j],t_1\big],\end{array}\right.$$
    then $\|f_j\|_{L^p(t_0,t_1;X)}=(t_1-t_0)\|v\|_X$, for every $j\in\mathbb{N}$, and for $m,n\in\mathbb{N}$ with $n>m$ we deduce that
    $$\|J_{t_0,t}^\alpha f_n-J_{t_0,t}^\alpha f_m\|_{L^{p/(1-p\alpha)}(t_0,t_1;X)}\geq\left[1-\left(\dfrac{m}{n}\right)^{1/p}\right]\dfrac{(t_1-t_0)^{1/p}}{\Gamma(\alpha+1)}.
    $$
In other words, we obtained that $\{f_j(t):j\in\mathbb{N}\}$ is bounded in $L^p(t_0,t_1;X)$ and $\{J_{t_0,t}^\alpha f_j(t):j\in\mathbb{N}\}$ does not have a cauchy subsequence in $L^{p/(1-p\alpha)}(t_0,t_1;X)$. Therefore, $\{J_{t_0,t}^\alpha f_j(t):j\in\mathbb{N}\}$ is not relatively compact in $L^{p/(1-p\alpha)}(t_0,t_1;X)$, what implies that $J_{t_0,t}^\alpha$ is not a compact operator. The source of this counter example is \cite[Remark 4.3.1]{GoVe1}.

\item[(ii)] Also, with a slight adaptation, we may obtain the same conclusion when $J_{t_0,t}^\alpha$ is viewed as a bounded operator from $L^p(t_0,\infty;X)$ into $L^{p/(1-p\alpha)}(t_0,\infty;X)$.
\end{itemize}

Here we prove our results about the compactness of RL fractional integral. Firstly, let us present some auxiliary notions and results.
\begin{definition} Assume that $\alpha\in(0,1)$ and $f:[t_0,t_1]\rightarrow \mathbb{R}$ is a given function.
\begin{itemize}
\item[(i)] The Riemann-Liouville (RL for short) fractional derivative of order $\alpha$ at $t_0$ of $f$ is defined by
\begin{equation}\label{fracinit00}D_{t_0,t}^\alpha f(t):=\dfrac{d}{dt}\big[J^{1-\alpha}_{t_0,t}f(t)\big],\end{equation}
for every $t\in [t_0,t_1]$, such that \eqref{fracinit00} exists. \vspace*{0.2cm}
\item[(ii)] The Caputo fractional derivative of order $\alpha$ at $t_0$ of $f$ is defined by
\begin{equation}\label{fracinit001}cD_{t_0,t}^\alpha f(t):=D_{t_0,t}^\alpha \big[f(t)-f(t_0)\big],\end{equation}
for every $t\in [t_0,t_1]$, such that \eqref{fracinit001} exists.
\end{itemize}
\end{definition}

\begin{remark}\label{diffcaputo} If $f:[t_0,t_1]\rightarrow \mathbb{R}$ is continuously differentiable in $[t_0,t_1]$, then $cD_{t_0,t}^\alpha f(t)$ is continuous in $[t_0,t_1]$ and it holds that
$$cD_{t_0,t}^\alpha f(t)=J^{1-\alpha}_{t_0,t}f^\prime(t),$$
for every $t\in[t_0,t_1]$. This fact is a classical consequence of differentiation under the integral sign.
\end{remark}

Now we recall Diethelm's mean value theorem with Caputo fractional derivative to real functions, which was proved in \cite{Die1,Die2}.

\begin{theorem}[Fractional Mean Value Theorem]\label{dithelmteo} Let $\alpha\in(0,1)$ and consider a continuous function $f:[t_0,t_1]\rightarrow\mathbb{R}$. If $cD_{t_0,t}^\alpha f(t)$ is also continuous in $[t_0,t_1]$, then there exists some $\xi\in(t_0,t_1)$ such that
$$\dfrac{f(t_1)-f(t_0)}{(t_1-t_0)^\alpha}=\left.\dfrac{cD_{t_0,t}^\alpha f(t)}{\Gamma(1+\alpha)}\right|_{t=\xi}.$$
\end{theorem}

To facilitate the proof of our forward theorems, let us address the following corollary, which is a natural consequence of Theorem \ref{dithelmteo}.

\begin{corollary}\label{meancapu} Consider $0<\beta<\alpha<1$ and $l,x>0$. Then there exists $\xi_x\in(0,l)$ and a constant $c_{\alpha,\beta}>0$ such that
$$\Big|(l+x)^{\alpha-1}-{x}^{\alpha-1}\Big|=l^\beta (\xi_x+x)^{\alpha-\beta-1}c_{\alpha,\beta}\int_{x/(\xi_x+x)}^1(1-w)^{-\beta}w^{\alpha-2}dw$$
\end{corollary}

\begin{proof} Consider function $\phi_x:[0,l]\rightarrow\mathbb{R}$ given by $\phi_x(t)=(t+x)^{\alpha-1}$. Since $\phi_x(t)$ is continuously differentiable in $[0,l]$, Remark \ref{diffcaputo} ensures that $cD_{0,t}^\beta\phi_x(t)$ is continuous in $[0,l]$ and that
\begin{multline*}cD_{0,t}^\beta\phi_x(t)=J^{1-\beta}_{t_0,t}\big[\phi_x(t)\big]^\prime=\dfrac{(\alpha-1)}{\Gamma(1-\beta)}\int_0^t(t-s)^{-\beta}(s+x)^{\alpha-2}ds\\
\stackrel[]{s=(t+x)w-x}{=}\dfrac{(\alpha-1)(t+x)^{\alpha-\beta-1}}{\Gamma(1-\beta)}\int_{x/(t+x)}^1(1-w)^{-\beta}w^{\alpha-2}dw,\end{multline*}
for every $t\in[0,l]$. But then, Theorem \ref{dithelmteo} ensures the existence of $\xi_x\in(0,l)$ such that
$$\Big|(l+x)^{\alpha-1}-{x}^{\alpha-1}\Big|=|\phi_x(l)-\phi_x(0)|=\left.\dfrac{l^\beta \big|cD_{0,t}^\beta \phi_x(t)\big|}{\Gamma(1+\beta)}\right|_{t=\xi_x},$$
and therefore
$$\Big|(l+x)^{\alpha-1}-{x}^{\alpha-1}\Big|=\dfrac{l^\beta(1-\alpha)(\xi_x+x)^{\alpha-\beta-1}}{\Gamma(1+\beta)\Gamma(1-\beta)}\int_{x/(\xi_x+x)}^1(1-w)^{-\beta}w^{\alpha-2}dw.$$

\end{proof}

Now we recall the classical Interpolation Theorem (see \cite[Theorem 5.1.1]{Ber1}).

\begin{theorem}[Interpolation Theorem]\label{interpolation} Let $1\leq p_1<p<p_2<\infty$ and assume that $f\in L^{p_1}(I;\mathbb{R})\cap L^{p_2}(I;\mathbb{R})$. Then $f\in L^{p}(I;\mathbb{R})$ with the following estimate holding
$$\|f\|_{L^{p}(I;\mathbb{R})}\leq\|f\|^{1-\theta}_{L^{p_1}(I;\mathbb{R})}\|f\|^\theta_{L^{p_2}(I;\mathbb{R})},$$
where $\theta\in(0,1)$ satisfies the identity
$$\dfrac{1}{p}=\dfrac{1-\theta}{p_1}+\dfrac{\theta}{p_2}.$$
\end{theorem}

Let us present the classical result of Simon about compactness in Bochner-Lebesgue spaces (see \cite[Theorem 1]{Si1} for details).

\begin{theorem}[Simon Compactness Theorem]\label{simon} Let $1\leq p<\infty$. A set of functions $F\subset L^p(t_0,t_1;X)$ is relatively compact in $L^p(t_0,t_1;X)$ if, and only if:\vspace*{0.2cm}
\begin{itemize}
\item[(i)] $\left\{\int_{t_0^*}^{t_1^*}f(t)dt\,:\,f\in F\right\}$ is relatively compact in $X$, for every $t_0<t_0^*<t_1^*<t_1$;\vspace*{0.2cm}
\item[(ii)] $\lim_{h\rightarrow 0^+}\left[\sup_{f\in F}\left(\int_{t_0}^{t_1-h}\|f(t+h)-f(t)\|_X^p\right)^{1/p}\right]=0$.
\end{itemize}
\end{theorem}

Bellow we present an intricate implication of the classical Minkowski's inequality to integrals (see \cite[Theorem 202]{HaLiPo1}), which we find important to give details of its proof.

\begin{lemma}\label{minkowski2} Assume that $f:[t_0,t_1]\times[t_0,t_1]\rightarrow\mathbb{R}$ is a Lebesgue measurable function and that $1\leq p<\infty$. If $h\in\big(0,(t_1-t_0)/2\big)$, then
\begin{multline*}\left[\int_{t_0}^{t_1-h}{\left|\int_{t}^{t+h}{f(t,s)}\,ds\right|^p}\,dt\right]^{1/p}\leq
\int_{t_0}^{t_0+h}\left[\int_{t_0}^{s}|{f}(t,s)|^p\,dt\right]^{1/p}\,ds\\
+\int_{t_0+h}^{t_1-h}\left[\int_{s-h}^{s}|{f}(t,s)|^p\,dt\right]^{1/p}\,ds
+\int_{t_1-h}^{t_1}\left[\int_{s-h}^{t_1-h}|{f}(t,s)|^p\,dt\right]^{1/p}\,ds. \end{multline*}
\end{lemma}

\begin{proof} Consider function $\tilde{f}:[t_0,t_1]\times[t_0,t_1]\rightarrow\mathbb{R}$ given by
$$\tilde{f}(t,s)=\left\{\begin{array}{ll}f(t,s),&\textrm{if }t\leq s\leq t+h,\vspace*{0.2cm}\\
0,&\textrm{otherwise}.\end{array}\right.$$

Then by the  classical Minkowski's inequality for integrals we have
\begin{multline}\label{finaldesiqul}\left[\int_{t_0}^{t_1-h}{\left|\int_{t}^{t+h}{f(t,s)}\,ds\right|^p}\,dt\right]^{1/p}=
\left[\int_{t_0}^{t_1-h}{\left|\int_{t_0}^{t_1}{\tilde{f}(t,s)}\,ds\right|^p}\,dt\right]^{1/p}\\
\leq
\int_{t_0}^{t_1}\left[\int_{t_0}^{t_1-h}|\tilde{f}(t,s)|^p\,dt\right]^{1/p}\,ds\leq
\int_{t_0}^{t_0+h}\left[\int_{t_0}^{t_1-h}|\tilde{f}(t,s)|^p\,dt\right]^{1/p}\,ds
\\+\int_{t_0+h}^{t_1-h}\left[\int_{t_0}^{t_1-h}|\tilde{f}(t,s)|^p\,dt\right]^{1/p}\,ds
+\int_{t_1-h}^{t_1}\left[\int_{t_0}^{t_1-h}|\tilde{f}(t,s)|^p\,dt\right]^{1/p}\,ds.\end{multline}

To complete the proof of this result, we need to reinterpret the integrals on the right side of the above inequality, as integrals of $f(t,s)$ instead of $\tilde{f}(t,s)$. In fact, Minkowski's inequality ensures that:
\begin{multline*}(i)\hspace*{0.2cm}\int_{t_0}^{t_0+h}\left[\int_{t_0}^{t_1-h}|\tilde{f}(t,s)|^p\,dt\right]^{1/p}\,ds
\leq\int_{t_0}^{t_0+h}\left[\int_{t_0}^{s}|\underbrace{\tilde{f}(t,s)}_{=f(t,s)}|^p\,dt\right]^{1/p}\,ds\\
+\int_{t_0}^{t_0+h}\left[\int_{s}^{t_1-h}|\underbrace{\tilde{f}(t,s)}_{=0}|^p\,dt\right]^{1/p}\,ds
=\int_{t_0}^{t_0+h}\left[\int_{t_0}^{s}|{f}(t,s)|^p\,dt\right]^{1/p}\,ds.\vspace*{0.1cm}\end{multline*}
\begin{multline*}(ii)\hspace*{0.2cm}\int_{t_0+h}^{t_1-h}\left[\int_{t_0}^{t_1-h}|\tilde{f}(t,s)|^p\,dt\right]^{1/p}\,ds
\leq\int_{t_0+h}^{t_1-h}\left[\int_{t_0}^{s-h}|\underbrace{\tilde{f}(t,s)}_{=0}|^p\,dt\right]^{1/p}\,ds\\
+\int_{t_0+h}^{t_1-h}\left[\int_{s-h}^{s}|\underbrace{\tilde{f}(t,s)}_{=f(t,s)}|^p\,dt\right]^{1/p}\,ds
+\int_{t_0+h}^{t_1-h}\left[\int_{s}^{t_1-h}|\underbrace{\tilde{f}(t,s)}_{=0}|^p\,dt\right]^{1/p}\,ds\\
=\int_{t_0+h}^{t_1-h}\left[\int_{s-h}^{s}|{f}(t,s)|^p\,dt\right]^{1/p}\,ds.\vspace*{0.1cm}\end{multline*}
\begin{multline*}(iii)\hspace*{0.2cm}\int_{t_1-h}^{t_1}\left[\int_{t_0}^{t_1-h}|\tilde{f}(t,s)|^p\,dt\right]^{1/p}\,ds
\leq\int_{t_1-h}^{t_1}\left[\int_{t_0}^{s-h}|\underbrace{\tilde{f}(t,s)}_{=0}|^p\,dt\right]^{1/p}\,ds\\
+\int_{t_1-h}^{t_1}\left[\int_{s-h}^{t_1-h}|\underbrace{\tilde{f}(t,s)}_{=f(t,s)}|^p\,dt\right]^{1/p}\,ds=\int_{t_1-h}^{t_1}\left[\int_{s-h}^{t_1-h}|{f}(t,s)|^p\,dt\right]^{1/p}\,ds.\vspace*{0.2cm}\end{multline*}

Inequality \eqref{finaldesiqul} together with itens $(i)$, $(ii)$ and $(iii)$ given above, complete the proof of the corollary.
\end{proof}

Finally we present a last important auxiliary result.

\begin{lemma}\label{compactlemma}
Let $p,q\in[1,\infty)$ and $\alpha>0$. If the operator
$$J_{t_0,t}^\alpha:L^p(t_0,t_1;X)\rightarrow L^{q}(t_0,t_1;X)$$
is compact, then for any bounded set $F\subset L^p(t_0,t_1;X)$ it holds that
\begin{equation}\label{novahip00}\left\{\begin{array}{l}\left\{J_{t_0,t_1^*}^{1+\alpha} f(t_1^*)-J_{t_0,t_0^*}^{1+\alpha} f(t_0^*)\,:\,f\in F\right\}\textrm{ is relatively compact in } X,\vspace*{0.2cm}\\
\textrm{for every }t_0<t_0^*<t_1^*<t_1.\end{array}\right.\end{equation}

\end{lemma}

It worths to emphasize that the proof of Lemma \ref{compactlemma} follows the same ideas presented in the case $q=p$, proved in \cite[Theorem 15]{CarFe0}. However, to facilitate the understanding of Theorems \ref{compactrieman00} and \ref{compactrieman01}, we prove it bellow.

\begin{proof1} Assume that $J_{t_0,t}^\alpha: L^p(t_0,t_1;X) \rightarrow L^q(t_0,t_1;X)$ is a compact operator, consider $t_0<t_0^*<t_1^*<t_1$ and define the linear operator %
$$\begin{array}{lclc}J_{t_0^*,t_1^*}:&L^q(t_0,t_1;X)&\rightarrow& X,\\
&f(s)&\mapsto& \int_{t_0^*}^{t_1^*}f(s)\,ds.\end{array}$$
Observe that $J_{t_0^*,t_1^*}$ is a bounded operator with
$$\left\|J_{t_0^*,t_1^*}\right\|_{\mathcal{L}(L^q(t_0,t_1;X),X)}\leq(t_1^*-t_0^*)^{1-(1/q)}.$$

Since we are assuming that $J_{t_0,t}^\alpha$ is compact, we deduce that
$$\begin{array}{lclc}J_{t_0^*,t_1^*}\circ J_{t_0,t}^\alpha:&L^p(t_0,t_1;X)&\rightarrow& X\\
&f(s)&\mapsto& \int_{t_0^*}^{t_1^*}J_{t_0,s}^\alpha f(s)\,ds.\end{array}$$
is a compact operator. However, notice that
\begin{multline*}\int_{t_0^*}^{t_1^*}J_{t_0,s}^\alpha f(s)\,ds=\dfrac{1}{\Gamma(\alpha)}\int_{t_0^*}^{t_1^*}\left[\int_{t_0}^s(s-w)^{\alpha-1} f(w)\,dw\right]\,ds\\
=\dfrac{1}{\Gamma(\alpha)}\int_{t_0^*}^{t_1^*}\left[\int_{t_0}^{t_0^*}(s-w)^{\alpha-1} f(w)\,dw\right]\,ds\\+\dfrac{1}{\Gamma(\alpha)}\int_{t_0^*}^{t_1^*}\left[\int_{t_0^*}^s(s-w)^{\alpha-1} f(w)\,dw\right]\,ds
\end{multline*}
and therefore the classical Fubini theorem of Bochner-Lebesgue functions \cite[Chapter X - Theorem
2]{Mik1} ensures that
\begin{multline*}\int_{t_0^*}^{t_1^*}J_{t_0,s}^\alpha f(s)\,ds=\dfrac{1}{\Gamma(\alpha)}\int_{t_0}^{t_0^*}\left[\int_{t_0^*}^{t_1^*}(s-w)^{\alpha-1} \,ds\right]f(w)\,dw\\+\dfrac{1}{\Gamma(\alpha)}\int_{t_0^*}^{t_1^*}\left[\int_{w}^{t_1^*}(s-w)^{\alpha-1}\,ds\right] f(w)\,dw,
\end{multline*}
what is equivalent to
$$\int_{t_0^*}^{t_1^*}J_{t_0,s}^\alpha f(s)\,ds=\dfrac{1}{\Gamma(\alpha+1)}\int_{t_0}^{t_1^*}(t_1^*-w)^{\alpha}f(w)\,dw-\dfrac{1}{\Gamma(\alpha+1)}\int_{t_0}^{t_0^*}(t_0^*-w)^{\alpha}f(w)\,dw.$$

Hence, if $F$ is any bounded subset of $L^p(t_0,t_1;X)$, we know that
$$\Big[J_{t_0^*,t_1^*}\circ J_{t_0,t}^\alpha \Big](F)=\left\{J_{t_0,t_1^*}^{1+\alpha} f(t_1^*)-J_{t_0,t_0^*}^{1+\alpha} f(t_0^*)\,:\,f\in F\right\}$$
is a relatively compact set in $X$, as we wanted.
\end{proof1}

Finally, let us prove our first compactness theorem.

\begin{theorem}\label{compactrieman00} Let $p\in[1,\infty)$ and $\alpha\in(0,1/p)$. If $q\in\big(p,p/(1-p\alpha)\big)$, the operator
$$J_{t_0,t}^\alpha:L^p(t_0,t_1;X)\rightarrow L^{q}(t_0,t_1;X)$$
is compact if, and only if, for any bounded set $F\subset L^p(t_0,t_1;X)$ it holds that
\begin{equation}\label{novahip00}\left\{\begin{array}{l}\left\{J_{t_0,t_1^*}^{1+\alpha} f(t_1^*)-J_{t_0,t_0^*}^{1+\alpha} f(t_0^*)\,:\,f\in F\right\}\textrm{ is relatively compact in } X,\vspace*{0.2cm}\\
\textrm{for every }t_0<t_0^*<t_1^*<t_1.\end{array}\right.\end{equation}
\end{theorem}

\begin{proof}
If we assume that $J_{t_0,t}^\alpha:L^p(t_0,t_1;X)\rightarrow L^{q}(t_0,t_1;X)$ is a compact operator and $F\subset L^p(t_0,t_1;X)$ is a bounded subset, then the desired conclusion follows directly from Lemma \ref{compactlemma}.

Conversely, let $F\subset L^p(t_0,t_1;X)$ be a bounded set. In order to prove that $J_{t_0,t}^\alpha (F)$ is relatively compact in $L^q(t_0,t_1;X)$, we need to verify that this set satisfies the conditions given in items $(i)$ and $(ii)$ of Theorem \ref{simon}.  Since
$$\left\{\int_{t_1^*}^{t_0^*}{J_{t_0,s}^\alpha f(s)}\,ds\,:\,f\in F\right\}=\left\{J_{t_0,t_1^*}^{1+\alpha} f(t_1^*)-J_{t_0,t_0^*}^{1+\alpha} f(t_0^*)\,:\,f\in F\right\},$$
item $(i)$ is a direct consequence of \eqref{novahip00}. Note that the last equality is a consequence of Fubini's Theorem (see \cite[Chapter X - Theorem $2$]{Mik1}), as justified in the proof of Lemma \ref{compactlemma}. Therefore, to complete the proof of this result we just need to verify that $J_{t_0,t}^\alpha (F)$ satisfies item $(ii)$ of Theorem \ref{simon}.

To this end, observe that for each $h>0$ (sufficiently small) and $s\in[t_0,t_1-h]$
\begin{multline*}\Gamma(\alpha)\big\|J_{t_0,s+h}^\alpha f(s+h)-J_{t_0,s}^\alpha f(s)\big\|_X\leq\int_{t_0}^s\big|(s-w)^{\alpha-1}-(s+h-w)^{\alpha-1}\big|\|f(w)\|_X\,dw\\
+\int_{s}^{s+h}(s+h-w)^{\alpha-1}\|f(w)\|_X\,dw.\end{multline*}

Therefore, Minkowski's inequality ensures that
\begin{multline*}\Gamma(\alpha)\left(\int_{t_0}^{t_1-h}\big\|J_{t_0,s+h}^\alpha f(s+h)-J_{t_0,s}^\alpha f(s)\big\|_X^q\,ds\right)^{1/q}\\\leq \left(\int_{t_0}^{t_1-h}\left[\int_{t_0}^s\big|(s-w)^{\alpha-1}-(s+h-w)^{\alpha-1}\big|\|f(w)\|_X\,dw\right]^q\,ds\right)^{1/q}\\
 +\left(\int_{t_0}^{t_1-h}\left[\int_{s}^{s+h}(s+h-w)^{\alpha-1}\|f(w)\|_X\,dw\right]^q\,ds\right)^{1/q}=:\big(\mathcal{I}_h+\mathcal{J}_h\big).\end{multline*}

In a first moment, let us prove that $\sup_{f\in F}\mathcal{I}_h\rightarrow0$, when $h\rightarrow0^+$. Consider $\beta\in(0,\alpha)$ given by
$$\beta=\alpha-\left[\dfrac{q-p}{qp}\right].$$
Then, if we set $x=s-w$, Corollary \ref{meancapu} ensures the existence of $\xi_{s,w}\in(0,h)$ and $c_{\alpha,\beta}>0$ such that
\begin{multline}\label{diethelm2}\big|(s-w)^{\alpha-1}-(s+h-w)^{\alpha-1}\big|\\=h^\beta (\xi_{s,w}+s-w)^{\alpha-\beta-1}c_{\alpha,\beta}\int_{(s-w)/[\xi_{s,w}+(s-w)]}^1(1-w)^{-\beta}w^{\alpha-2}dw,\end{multline}
for every $t_0<w<s<t_1-h$. Consider now function $\delta:[0,1]\rightarrow\mathbb{R}$ given by
$$\delta(t)=\left\{\begin{array}{ll}t^{\beta+1-\alpha}\displaystyle\int_{t}^1(1-w)^{-\beta}w^{\alpha-2}\,dw,&\textrm{ if }t\in(0,1),\vspace*{0.2cm}\\0,&\textrm{ if } t=0\textrm{ or }1.\end{array}\right.$$
Observe that $\delta(t)$ is continuous in $[0,1]$, since  L'Hôpital's rule ensures that
$$\lim_{t\rightarrow0^{+}}\delta(t)=\lim_{t\rightarrow0^{+}}\left(\dfrac{\displaystyle\int_{t}^1(1-w)^{-\beta}w^{\alpha-2}\,dw}{t^{\alpha-\beta-1}}\right)=\lim_{t\rightarrow0^{+}}\dfrac{-(1-t)^{-\beta}t^{\alpha-2}}{(\alpha-\beta-1)t^{\alpha-\beta-2}}=0$$
and the fact that $w^{\alpha-2}<t^{\alpha-2}$ guarantees that
$$\lim_{t\rightarrow1^{-}}\delta(t)\leq\lim_{t\rightarrow1^{-}}t^{\beta-1}\displaystyle\int_{t}^1(1-w)^{-\beta}\,dw=\lim_{t\rightarrow1^{-}}\dfrac{t^{\beta-1}(1-t)^{1-\beta}}{1-\beta}=0.$$

Thus, there exists $M_{\alpha,\beta}>0$ such that $|\delta(t)|\leq M_{\alpha,\beta}$, for every $t\in[0,1]$. With this, by considering identity \eqref{diethelm2}, we obtain
\begin{multline*}\mathcal{I}_h={h^\beta}{c_{\alpha,\beta}}\left(\int_{t_0}^{t_1-h}\left[\int_{t_0}^s \delta\left(\dfrac{s-w}{\xi_{s,w}+s-w}\right)(s-w)^{\alpha-\beta-1}\|f(w)\|_X\,dw\right]^q\,ds\right)^{1/q}\\
\leq {h^\beta}{c_{\alpha,\beta}}M_{\alpha,\beta}\big\|J^{\alpha-\beta}_{t_0,s}\|f(s)\|_X\big\|_{L^q(t_0,t_1;\mathbb{R})}.\end{multline*}

Since $q=p/[1-p(\alpha-\beta)]$ and $0<\alpha-\beta<\alpha<1/p$, Theorem \ref{theoHLBoch} ensures that
$$\mathcal{I}_h\leq {h^\beta}\underbrace{{c_{\alpha,\beta}}M_{\alpha,\beta}C_{\alpha-\beta,p}}_{K_{\alpha,\beta,p}}\|f(t)\|_{L^p(t_0,t_1;X)}\leq {h^\beta} K_{\alpha,\beta,p}\|F\|_{L^p(t_0,t_1;X)},$$
where $\|F\|_{L^p(t_0,t_1;X)}:=\sup_{f\in F}\|f\|_{L^p(t_0,t_1;X)}<\infty$. The conclusion follows now directly from the estimate
$$\lim_{h\rightarrow0^+}\left[\,\sup_{f\in F}\mathcal{I}_h\,\right]\leq \lim_{h\rightarrow0^+}{h^\beta} K_{\alpha,\beta,p}\|F\|_{L^p(t_0,t_1;X)}=0.$$

To complete the proof of this theorem, let us verify that $\sup_{f\in F}\mathcal{J}_h\rightarrow0$, when $h\rightarrow0^+$. Choose $k\in\big(1,\min\{{p,1/(1-\alpha)}\}\big)$ and $\theta\in(0,1)$ given by
$$\theta=\dfrac{p}{q}\left[\dfrac{q-k}{p-k(1-p\alpha)}\right].$$
Since $1<k<q<p/(1-p\alpha)<\infty$, Theorem \ref{interpolation} gives us
\begin{multline*}\mathcal{J}_h\leq\left(\int_{t_0}^{t_1-h}\left[\int_{s}^{s+h}(s+h-w)^{\alpha-1}\|f(w)\|_X\,dw\right]^k\,ds\right)^{(1-\theta)/k}\\
\times\left(\int_{t_0}^{t_1-h}\left[\int_{s}^{s+h}(s+h-w)^{\alpha-1}\|f(w)\|_X\,dw\right]^{ p/(1-p\alpha)}\,ds\right)^{\theta(1-p\alpha)/p}\\=:\mathcal{K}_h\times\mathcal{L}_h.\hspace*{1cm}\end{multline*}

Now observe that Lemma \ref{minkowski2} ensures that
\begin{multline*}{\mathcal{K}_h}^{1/(1-\theta)}\leq \int_{t_0}^{t_0+h}\left[\int_{t_0}^{w}(s+h-w)^{k(\alpha-1)}ds\right]^{1/k}\|f(w)\|_X\,dw\\
+\int_{t_0+h}^{t_1-h}\left[\int_{w-h}^{w}(s+h-w)^{k(\alpha-1)}ds\right]^{1/k}\|f(w)\|_X\,dw\\
+\int_{t_1-h}^{t_1}\left[\int_{w-h}^{t_1-h}(s+h-w)^{k(\alpha-1)}ds\right]^{1/k}\|f(w)\|_X\,dw,\end{multline*}
and since $k(\alpha-1)+1>0$, we obtain
\begin{multline*}{\mathcal{K}_h}^{1/(1-\theta)}\leq \int_{t_0}^{t_0+h}\left[\dfrac{h^{k(\alpha-1)+1}-(t_0+h-w)^{k(\alpha-1)+1}}{k(\alpha-1)+1}\right]^{1/k}\|f(w)\|_X\,dw\\
+\int_{t_0+h}^{t_1-h}\left[\dfrac{h^{k(\alpha-1)+1}}{k(\alpha-1)+1}\right]^{1/k}\|f(w)\|_X\,dw\\
+\int_{t_1-h}^{t_1}\left[\dfrac{(t_1-w)^{k(\alpha-1)+1}}{k(\alpha-1)+1}\right]^{1/k}\|f(w)\|_X\,dw.\end{multline*}

But then we have
$${\mathcal{K}_h}^{1/(1-\theta)}\leq \left(\dfrac{h^{(\alpha-1)+(1/k)}}{[k(\alpha-1)+1]^{1/k}}\right)\int_{t_0}^{t_1}\|f(w)\|_X\,dw,$$
what, by using H\"{o}lder's inequality, gives us
$${\mathcal{K}_h}^{1/(1-\theta)}\leq
\left(\dfrac{h^{(\alpha-1)+(1/k)}(t_1-t_0)^{1-(1/p)}}{[k(\alpha-1)+1]^{1/k}}\right)\|f\|_{L^p(t_0,t_1;X)}.$$

Thus, we deduce that
$${\mathcal{K}_h}\leq \left(\dfrac{h^{(\alpha-1)+(1/k)}(t_1-t_0)^{1-(1/p)}}{[k(\alpha-1)+1]^{1/k}}\right)^{1-\theta}\|F\|^{1-\theta}_{L^p(t_0,t_1;X)},$$
where, like before, $\|F\|_{L^p(t_0,t_1;X)}:=\sup_{f\in F}\|f\|_{L^p(t_0,t_1;X)}<\infty$.

On the other hand,
\begin{multline*}\mathcal{L}_h= \left(\int_{t_0}^{t_1-h}\left[\int_{t_0}^{s+h}(s+h-w)^{\alpha-1}\|f(w)\|_X\,dw\right]^{ p/(1-p\alpha)}\,ds\right)^{\theta(1-p\alpha)/p}\\
\stackrel[]{s=r-h}{=}\left(\int_{t_0+h}^{t_1}\left[\int_{t_0}^{r}(r-w)^{\alpha-1}\|f(w)\|_X\,dw\right]^{ p/(1-p\alpha)}\,dr\right)^{\theta(1-p\alpha)/p},\end{multline*}
what implies
$$\mathcal{L}_h= \big[\Gamma(\alpha)\big]^\theta\|J^\alpha_{t_0,r}\|f\|_X\|^\theta_{L^{p/(1-p\alpha)}(t_0+h,t_1;\mathbb{R})}\leq \big[\Gamma(\alpha)\big]^\theta\|J^\alpha_{t_0,r}\|f\|_X\|^\theta_{L^{p/(1-p\alpha)}(t_0,t_1;\mathbb{R})}.$$

Thus Theorem \ref{theoHLBoch} ensures that
$$\mathcal{L}_h\leq \big[\Gamma(\alpha)C_{\alpha,p}\big]^\theta\|f\|^\theta_{L^{p}(t_0,t_1;X)}\leq \big[\Gamma(\alpha)C_{\alpha,p}\big]^\theta\|F\|^\theta_{L^p(t_0,t_1;X)},$$
again, like before, $\|F\|_{L^p(t_0,t_1;X)}:=\sup_{f\in F}\|f\|_{L^p(t_0,t_1;X)}<\infty$.

The above computations allow us to deduce
$$\sup_{f\in F}\mathcal{J}_h\leq  \left\{\dfrac{h^{(\alpha-1)+(1/k)}(t_1-t_0)^{1-(1/p)}}{[k(\alpha-1)+1]^{1/k}}\right\}^{1-\theta}\big[\Gamma(\alpha)C_{\alpha,p}\big]^\theta\|F\|_{L^p(t_0,t_1;X)},$$
what implies that $\sup_{f\in F}\mathcal{J}_h\rightarrow0$, when $h\rightarrow0^+$.
\end{proof}

Recall now that the case $q=p$ was already discussed in a previous work; see Theorem \ref{compactrieman}. Hence, in what follows we present our next result which improves the conclusions of Theorems \ref{compactrieman} and \ref{compactrieman00} a little bit more.

\begin{theorem}\label{compactrieman01} Let $p\in[1,\infty)$ and $\alpha>0$. If $q\in\big[1,p)$, the operator
$$J_{t_0,t}^\alpha:L^p(t_0,t_1;X)\rightarrow L^{q}(t_0,t_1;X)$$
is compact if, and only if, for any bounded set $F\subset L^p(t_0,t_1;X)$ it holds that
\begin{equation*}\left|\begin{array}{l}\left\{J_{t_0,t_1^*}^{1+\alpha} f(t_1^*)-J_{t_0,t_0^*}^{1+\alpha} f(t_0^*)\,:\,f\in F\right\}\textrm{ is relatively compact in } X,\vspace*{0.2cm}\\
\textrm{for every }t_0<t_0^*<t_1^*<t_1.\end{array}\right.\end{equation*}
\end{theorem}

\begin{proof} The first part follows from Lemma \ref{compactlemma}.

Conversely, since we already know that $J^\alpha_{t_0,t}: L^p(t_0,t_1;X)\rightarrow L^p(t_0,t_1;X)$ is compact (see Theorem \ref{compactrieman}) and that the embedding operator
$$\begin{array}{lclc}E:&L^p(t_0,t_1;X)&\rightarrow& L^q(t_0,t_1;X),\vspace*{0.0cm}\\
&f(t)&\mapsto& f(t),\end{array}$$
is continuous, we conclude that $E\circ J^\alpha_{t_0,t}:L^p(t_0,t_1;X)\rightarrow L^q(t_0,t_1;X)$ is a compact operator, as we wanted.
\end{proof}

We can summarize all the above compactness results in the following theorem.

\begin{theorem}\label{compactriemanfinal} Assume that one of the following assertions hold:
\begin{itemize}
\item[(i)] $p\in[1,\infty)$, $\alpha\in(0,\infty)$ and $q\in[1,p]$;\vspace*{0.2cm}
\item[(ii)] $p\in(1,\infty)$, $\alpha\in(0,1/p)$ and $q\in\big(p,p/(1-p\alpha)\big)$;\vspace*{0.2cm}
\item[(iii)] $p=1$, $\alpha\in(0,1)$ and $q\in\big[1,1/(1-\alpha)\big)$. \vspace*{0.2cm}
\end{itemize}
Then $J_{t_0,t}^\alpha:L^p(t_0,t_1;X)\rightarrow L^{q}(t_0,t_1;X)$ is a compact operator if, and only if, for any bounded set $F\subset L^p(t_0,t_1;X)$ it holds that
\begin{equation*}\left|\begin{array}{l}\left\{J_{t_0,t_1^*}^{1+\alpha} f(t_1^*)-J_{t_0,t_0^*}^{1+\alpha} f(t_0^*)\,:\,f\in F\right\}\textrm{ is relatively compact in } X,\vspace*{0.2cm}\\
\textrm{for every }t_0<t_0^*<t_1^*<t_1.\end{array}\right.\end{equation*}

\end{theorem}

A natural corollary of the above result to the finite dimensional case is the following:

\begin{corollary} Assume that $p$, $\alpha$ and $q$ satisfies $(i)$, $(ii)$ or $(iii)$ of Theorem \ref{compactriemanfinal}. If $X$ is a finite dimensional Banach space, then $J_{t_0,t}^\alpha:L^p(t_0,t_1;X)\rightarrow L^{q}(t_0,t_1;X)$ is a compact operator.
\end{corollary}

\begin{proof} Just observe that for $f\in L^p(t_0,t_1;X)$ and $t^*\in(t_0,t_1)$, H\"{o}lder's inequality ensures that
\begin{multline*}\left\|\int_{t_0}^{t^*}(t^*-s)^{\alpha}f(s)\,ds\right\|_X\leq
\left[\int_{t_0}^{t^*}(t^*-s)^{(\alpha p)/(p-1)}\,ds\right]^{(p-1)/p}\|f\|_{L^p(t_0,t_1;X)}\\=\left[\dfrac{p-1}{(\alpha+1)p-1}\right]^{(p-1)/p}(t^*-t_0)^{\alpha+1-(1/p)}\|f\|_{L^p(t_0,t_1;X)}.\end{multline*}

The above estimate allows us to conclude that for any $F\subset L^p(t_0,t_1;X)$ bounded, the set $\{J_{t_0,t_1^*}^{1+\alpha} f(t_1^*)-J_{t_0,t_0^*}^{1+\alpha} f(t_0^*)\,:\,f\in F\}$ is also bounded in $X$. Since the dimension of $X$ is finite, then we deduce that %
$$\{J_{t_0,t_1^*}^{1+\alpha} f(t_1^*)-J_{t_0,t_0^*}^{1+\alpha} f(t_0^*)\,:\,f\in F\}$$
is relatively compact in $X$. Thus, Theorem \ref{compactriemanfinal} gives us the compactness of the RL fractional integral operator from $L^p(t_0,t_1;X)$ into $L^q(t_0,t_1;X)$.
\end{proof}

\end{document}